\documentclass[12pt]{amsart}

\usepackage{amssymb,verbatim}
\usepackage{amsmath,amsfonts}
\usepackage[mathscr]{euscript} 
\usepackage{amsthm}
\usepackage{color}
\usepackage{url}
\usepackage{graphicx} 
\usepackage{enumitem} 
\usepackage{lineno,xcolor} 
\usepackage{hyperref} 
\hypersetup{colorlinks} 
\usepackage{bbm} 
\usepackage{mathrsfs} 

\usepackage{graphicx} 

\usepackage{standalone} 

\setlist[enumerate]{
label=\textnormal{({\roman*})},
ref={\roman*}}

\makeatletter
\def\th@plain{%
  \thm@notefont{}
  \itshape 
}
\def\th@definition{%
  \thm@notefont{}
  \normalfont 
}
\makeatother

\makeatletter
\newtheorem*{rep@theorem}{\rep@title}
\newcommand{\newreptheorem}[2]{%
\newenvironment{rep#1}[1]{%
 \def\rep@title{#2 \ref{##1}}%
 \begin{rep@theorem}}%
 {\end{rep@theorem}}}
\makeatother

\newreptheorem{theorem}{Theorem}
\newreptheorem{corollary}{Corollary}

\newtheorem{theorem}{Theorem}[section]
\newtheorem{lemma}[theorem]{Lemma}
\newtheorem{proposition}[theorem]{Proposition}

\newtheorem{corollary}[theorem]{Corollary}
\newtheorem{conjecture}[theorem]{Conjecture}

\newtheorem{problem}[theorem]{Problem}

\theoremstyle{remark}
\newtheorem*{remark}{Remark}

\theoremstyle{definition}

\newtheorem{definition}[theorem]{Definition}
  
\makeatletter
\newcommand*{\da@rightarrow}{\mathchar"0\hexnumber@\symAMSa 4B }
\newcommand*{\da@leftarrow}{\mathchar"0\hexnumber@\symAMSa 4C }
\newcommand*{\xdashrightarrow}[2][]{%
  \mathrel{%
    \mathpalette{\da@xarrow{#1}{#2}{}\da@rightarrow{\,}{}}{}%
  }%
}
\newcommand{\xdashleftarrow}[2][]{%
  \mathrel{%
    \mathpalette{\da@xarrow{#1}{#2}\da@leftarrow{}{}{\,}}{}%
  }%
}
\newcommand*{\da@xarrow}[7]{%
  \sbox0{$\ifx#7\scriptstyle\scriptscriptstyle\else\scriptstyle\fi#5#1#6\m@th$}%
  \sbox2{$\ifx#7\scriptstyle\scriptscriptstyle\else\scriptstyle\fi#5#2#6\m@th$}%
  \sbox4{$#7\dabar@\m@th$}%
  \dimen@=\wd0 %
  \ifdim\wd2 >\dimen@
    \dimen@=\wd2 %
  \fi
  \count@=2 %
  \def\da@bars{\dabar@\dabar@}%
  \@whiledim\count@\wd4<\dimen@\do{%
    \advance\count@\@ne
    \expandafter\def\expandafter\da@bars\expandafter{%
      \da@bars
      \dabar@ 
    }%
  }%
  \mathrel{#3}%
  \mathrel{%
    \mathop{\da@bars}\limits
    \ifx\\#1\\%
    \else
      _{\copy0}%
    \fi
    \ifx\\#2\\%
    \else
      ^{\copy2}%
    \fi
  }%
  \mathrel{#4}%
}
\makeatother

\makeatletter 
\newcommand{\overrightharpoon}{%
  \mathpalette{\overarrow@\rightharpoonfill@}}
\def\rightharpoonfill@{\arrowfill@\relbar\relbar\rightharpoonup}
\makeatother

\makeatletter 
\newcommand{\osh}{\mathpalette{\overarrowsmall@\rightharpoonfill@}}
\def\rightharpoonfill@{\arrowfill@\relbar\relbar\rightharpoonup}
\newcommand{\overarrowsmall@}[3]{%
  \vbox{%
    \ialign{%
      ##\crcr
      #1{\smaller@style{#2}}\crcr
      \noalign{\nointerlineskip}%
      $\m@th\hfil#2#3\hfil$\crcr
    }%
  }%
}
\def\smaller@style#1{%
  \ifx#1\displaystyle\scriptstyle\else
    \ifx#1\textstyle\scriptstyle\else
      \scriptscriptstyle
    \fi
  \fi
}
\makeatother

\makeatletter
\newcommand{\mylabel}[2]{#2\def\@currentlabel{#2}\label{#1}}
\makeatother

\DeclareMathOperator{\Ac}{\mathcal{A}} 

\DeclareMathOperator{\Dc}{\mathcal{D}}
\DeclareMathOperator{\Eb}{\mathbb{E}} 

\DeclareMathOperator{\Fc}{\mathcal{F}} 
\DeclareMathOperator{\ft}{\mathsf{f}} 
\DeclareMathOperator{\Hc}{\mathcal{H}} 
\DeclareMathOperator{\Hor}{\mathcal{H}} 
\DeclareMathOperator{\IID}{\mathsf{IID}} 
\DeclareMathOperator{\IUD}{\mathsf{IUD}} 
\DeclareMathOperator{\MC}{\mathcal{M}} 
\DeclareMathOperator{\Nb}{\mathbb{N}} 
\DeclareMathOperator{\Ngh}{{Ngh}} 
\DeclareMathOperator{\Oc}{\mathcal{O}} 
\DeclareMathOperator{\Pb}{\mathbb{P}} 
\DeclareMathOperator{\Rb}{\mathbb{R}} 
\DeclareMathOperator{\rec}{\textnormal{rec}} 
\DeclareMathOperator{\satu}{\mathbbm{1}} 
\DeclareMathOperator{\Stack}{\mathsf{Stack}} 
\DeclareMathOperator{\Ver}{\mathcal{V}} 
\DeclareMathOperator{\xb}{\mathbf{x}} 
\DeclareMathOperator{\Xb}{\mathbf{X}} 
\DeclareMathOperator{\Yb}{\mathbf{Y}} 
\DeclareMathOperator{\wt}{\mathsf{wt}} 
\DeclareMathOperator{\Zb}{\mathbb{Z}} 

\def\.{\hskip.06cm}
\def\ts{\hskip.03cm}

\title{Recurrence of horizontal-vertical walks}
\author{Swee Hong Chan}
\date{\today}
 \address{Department of Mathematics, UCLA, Los Angeles, CA.}
\email{\url{sweehong@math.ucla.edu}}

 \thanks{Department of Mathematics, UCLA. Partially supported by NSF grant DMS-1455272.}

\begin{document}
 	
\begin{abstract}
	Consider a  nearest neighbor random walk on the two-dimensional integer lattice,
	where each vertex is initially labeled  either `H' or `V', uniformly and independently.
	At each discrete time step, the walker resamples the label at its current location (changing `H' to `V' and `V' to `H' with probability $q$).
	Then, it takes a mean zero horizontal step if the new label is `H', and a mean zero vertical step if the new label is `V'.  
	This model is a randomized version of the deterministic rotor walk, for which 
	its recurrence (i.e., visiting every vertex infinitely often with probability 1) in two dimensions is still an  open problem.
	We answer the analogous question for  the horizontal-vertical walk, by showing that the horizontal-vertical walk  is recurrent   for  $q \in (\frac{1}{3},1]$.
\end{abstract}


 \maketitle

\section{Introduction}\label{secintro}

In an $\Hor$--$\Ver$ walk,
each vertex of $\Zb^2$ is labeled  either $\Hor$ or $\Ver$,  with $n$ walkers initially located at the origin.  At each discrete time step, a walker is chosen following a cyclic turn order.
This walker resamples the label at its current location (changing $\Hor$ to $\Ver$, and $\Ver$  to $\Hor$, with probability $q \in (0,1]$, independent of the past), and then takes a mean zero horizontal step if the new label is $\Hor$,  and a mean zero vertical step if the new label is $\Ver$.  
We will show that this walk is \emph{recurrent} (i.e., every vertex  gets visited infinitely many times a.s.) under various scenarios~(see Theorem~\ref{thmrecurrence many walker}--\ref{thmrecurrence IID} below).


Our study of $\Hor$--$\Ver$ walks is motivated by \emph{rotor walks}, which are   deterministic versions of the simple random walk:
Each vertex of $\Zb^2$ is labeled with an arrow that is pointing to one of its neighbors.
At each discrete time step, a walker turns the arrow at its current location 90-degree clockwise, and the walker moves to the neighbor specified by the new arrow.
It is a longstanding open problem~\cite{PDDK} to determine if the rotor walk with a single walker, with the initial arrow at each vertex  $x$ being independent and pointing to a uniform random neighbor of $x$, is recurrent.
It is thus slightly  surprising that the analogous result can be proved for $\Hor$--$\Ver$ walks (see Theorem~\ref{thmrecurrence IID}), which can be attributed to the extra randomness in  
$\Hor$--$\Ver$ walks (compare with \cite{GMV}, where   extra randomness helps in making the recurrence problem for the mirror model analyzable).


{\renewcommand{\arraystretch}{-10}
	\begin{figure}[hbt!]
		\begin{tabular}{c c c}
			\hspace{-0.5 cm}
			\includegraphics[scale=0.38]{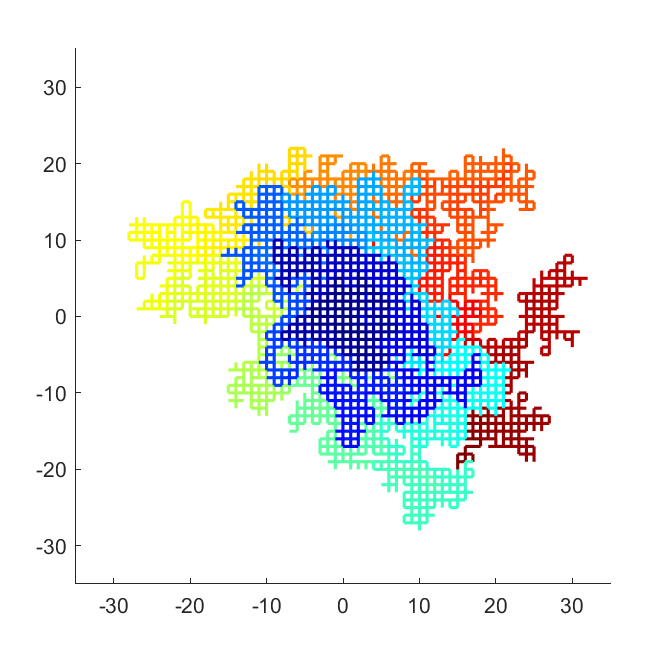} & \hspace{-0.25cm}  
			\includegraphics[scale=0.38]{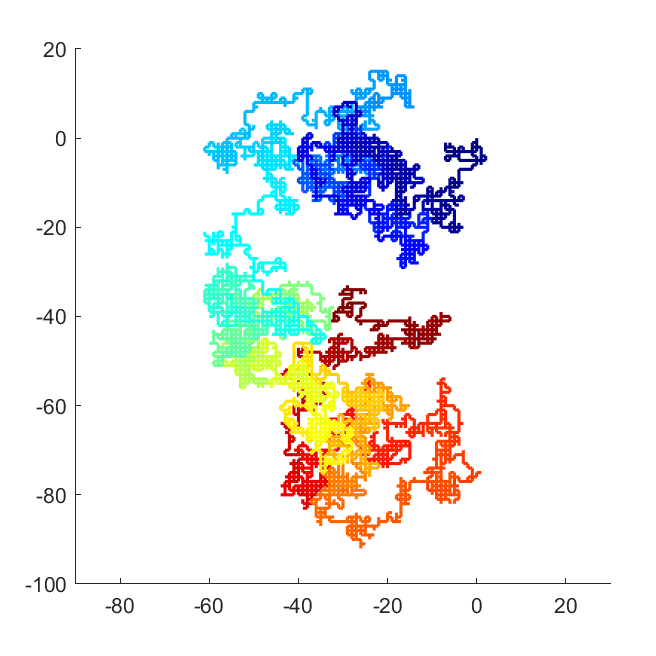} & \hspace{-0.25cm} 
			\includegraphics[scale=0.38]{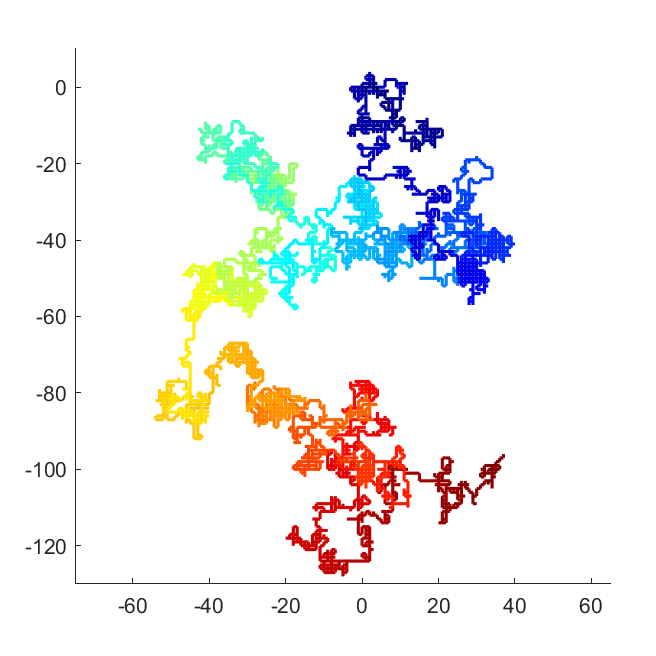}
		\end{tabular}
		\caption{The vertices  visited by the 10,000-step single walker version of   the   
			rotor walk (left),  the $\Hor$--$\Ver$ walk with $q=1$ (middle), and  the simple random walk (right) on $\mathbb{Z}^2$;
			these processes are ordered in increasing amount of randomness.
			Each edge is colored according to the time of its first visit by the walker. }
		\label{figure: simple random walk vs rotor walk}%
	\end{figure}
}

The one-dimensional counterpart of $\Hor$--$\Ver$ walks, called \emph{$p$-rotor walks} on $\Zb$, was  studied in~\cite{HLSH},
where each vertex is labeled $L$ (left) or $R$ (right), and is changed to the opposite label with probability $p$ whenever the vertex is visited.
The $p$-rotor walk is  a special case of  \emph{excited random walks with Markovian cookie stacks}~\cite{KP}, where the labels  evolve following the transition rules of a prescribed Markov chain. 
The recurrence and transience of these two models were studied in both works and are now completely understood. 
On the other hand, until recently, very little was known for the recurrence and transience in higher dimensions.
From this perspective, 
this paper aims to begin extending their works to   higher dimensions, for which some standard methods for $\Zb$ (e.g., generalized Ray-Knight theory~\cite{Tot}) cannot be applied anymore.

In analyzing $\Hor$--$\Ver$ walks, we take our inspiration from the theory of random walks in random environment~(see e.g., \cite{Zei,Szn}), in which the environment (i.e., the labels at each vertex) affects the motion of the walker, but the walker does not affect the environment. 
In contrast,  with $\Hor$--$\Ver$ walks, the environment evolves \emph{in tandem with} the motion of the walkers.
This   necessitates a different approach for proving  recurrence, 
as  common approaches in    random walks in random environments (e.g., Nash-Williams inequality, see~\cite[Lemma~A.2]{Zei}) cannot be applied to nonstatic environments.

\medskip

\subsection{Main results}
We now present the main results of this paper.

\smallskip

\begin{theorem}\label{thmrecurrence many walker}
	Let \. $q >0$ \.  and \. $n \geq \lfloor \frac{|4q-2|}{q} \rfloor + 1$\..
	Then, for every choice of the initial labels at each vertex, the $\Hor$--$\Ver$ walk with $n$ walkers   is recurrent.
\end{theorem}

\smallskip

Described in words, excepting the case $q=0$ (i.e., when the environment is not altered by the walkers),
\. $\Hor$--$\Ver$ walks can always be made recurrent by adding a fixed number of walkers.
In this respect Theorem~\ref{thmrecurrence many walker}  is the best outcome one could hope for as  for $q=0$ one can choose the initial labels  so that the $\Hor$--$\Ver$ walk is never recurrent regardless of the number of walkers (see Section~\ref{subsecq0}).

As a consequence of Theorem~\ref{thmrecurrence many walker}, we  have the following corollary  for $\Hor$--$\Ver$ walks with a single walker.


\smallskip

\begin{corollary}\label{correcurrence single walker}
	Let $q \in (\frac{2}{5},\frac{2}{3})$.
	Then, for every choice of the initial labels at each vertex, the $\Hor$--$\Ver$ walk with a single walker  is recurrent.
\end{corollary}

\smallskip

When the labels are sampled from the $\IID$ uniform measure on $\{\Hor, \Ver\}$, 
the recurrence regime in Corollary~\ref{correcurrence single walker} can be expanded even further.


\smallskip

\begin{theorem}\label{thmrecurrence IID}
	Let $q \in (\frac{1}{3},1]$, and let the initial labels be drawn independently and uniformly from \. $\{\Hor, \Ver\}$\..
	Then
	the  $\Hor$--$\Ver$ walk with a single walker is recurrent a.s..
\end{theorem}

\smallskip

The case $q=1$ in Theorem~\ref{thmrecurrence IID} deserves a special mention, as  this $\Hor$--$\Ver$ walk is  \emph{non}-elliptic, a property that also applies to rotor walks.
 (A walk is \emph{elliptic} if, every neighbor of the current location of the walker, is visited at the next step with positive probability.)
Most recurrence and transience results in this area (e.g., \cite{KOS,KP,PT}) assume some versions of ellipticity,
and thus Theorem~\ref{thmrecurrence IID} holds the  distinction of being one of the few results in this area for non-elliptic walks. 

\smallskip

In the proof of Theorem~\ref{thmrecurrence many walker} and Theorem~\ref{thmrecurrence IID}, we use   martingales that track both the number of departures from the origin and the labels at each vertex at any given time~(see Definition~\ref{defmartingale}), combined with zero-one laws that we develop for the recurrence of $\Hor$--$\Ver$ walks (see Section~\ref{subseczero-one laws}).
By applying the optional stopping theorem to the martingales,
we derive  lower bounds for  the return probabilities $p_k$  (i.e., the probability  that the  number of returns to  the origin is at least $k$).
These lower bounds turn out to be uniformly bounded away from $0$ under the hypotheses of Theorem~\ref{thmrecurrence many walker},
 and we then use a zero-one law~(see Proposition~\ref{proposition: zero-one law recurrence}) to conclude that the walk is recurrent.

When the labels are sampled from the $\IID$ uniform measure on $\{\Hor, \Ver\}$, we derive  improved lower bounds~(see~\eqref{eqlower bound IID}), which are again uniformly  bounded away from $0$
 when $q$ is contained in the interval $(\frac{1}{3},1)$ (note that this regime does not include $q=1$).
 However, 
 for the case $q=1$,
 these improved lower bounds yield only the trivial lower bound,
 and thus we need to improve them even further.
 We achieve this further improvement by proving an anti-concentration inequality, 
 for the probability of the label of an arbitrary vertex being equal to a given label (see Lemma~\ref{lemRadon Nikodym 2}).
 Finally,  we use another 
  zero-one law~(see Proposition~\ref{proposition: zero-one law IUD}) to conclude that the walk is recurrent.

\medskip

\subsection{Comparison with previous works}

The  martingale approach in this paper  dates back to the work of Holroyd and Propp~\cite{HP}, 
and since then has been successfully applied to prove various results for rotor walks (see e.g., \cite{HS,FGLP,Cha2}) and $p$-rotor walks (see e.g.,\cite{HLSH}).
The zero-one laws for the recurrence  of $\Hor$--$\Ver$ walks are inspired by zero-one laws for the directional transience of excited cookie random walks (see  \cite{KZ} and references therein), but with proofs based on the work of  \cite{AH2} originally developed for rotor walks.
The anti-concentration inequality, used in the proof of Theorem~\ref{thmrecurrence IID} for $q=1$, is original to this paper, to the best of the author's knowledge.

\medskip

\subsection{Related works}
$\Hor$--$\Ver$ walks are randomized versions of rotor walks (discovered independently by \cite{PDDK,WLB,Pro}),  where the last exit from each vertex follows an assigned deterministic order.
Ander and Holroyd~\cite{AH2} showed that,
one can always find a labeling for vertices of  $\Zb^d$ ($d\geq 1$) such that the  (single walker) rotor walk is recurrent.
On the other end of the spectrum, Florescu et al.~\cite{FGLP} and Chan~\cite{Cha1} constructed different labelings for vertices of $\Zb^d$ ($d\geq 3$) for which the rotor walk is transient regardless of the number of walkers.
The recurrence and transience of rotor walks have also been investigated for other graphs, including  regular trees~\cite{LL,AH1,MO}, Galton-Watson trees~\cite{HMSH}, directed cover of graphs~\cite{HS2}, and oriented lattices~\cite{FLP}.
Notably, the recurrence of  rotor walks on $\Zb^2$ with labels sampled from the  $\IID$ uniform measure on $\{\text{up, down, left, right}\}$ remains an open problem despite numerous investigations of this subject.
We refer the reader to \cite{HLM} for an excellent exposition on rotor walks and related subjects.

\smallskip

$\Hor$--$\Ver$ walks were introduced in \cite{CGLL} as  the representative example of \emph{random walks with local memory}, or RWLM for short,
where each vertex stores one bit of information by remembering the last exit from the vertex,
and the next exit is determined by the transition rule of a local Markov chain assigned to the vertex.
We discuss RWLMs in more details in Section~\ref{secpreliminaries},
where we derive zero-one laws for their recurrence.

One dimensional RWLMs  are more commonly studied 
in the literature under the name \emph{excited random walks}, or \emph{cookie random walks}:
 A pile of cookies is initially placed at each vertex of $\Zb$.
 Upon visiting  a vertex,
the walker consumes the topmost cookie from the pile  and moves one step to the right
or to the left with probabilities prescribed by that cookie.
If there are no cookies left at the current vertex,
the walker moves one step to the right or to the left with equal probabilities.
A cookie is \emph{positive} (resp. \emph{negative}) if its consumption 
 makes the walker moves 
right with probability larger (resp. smaller) than $\frac{1}{2}$.
The recurrence and transience of cookie random walks on $\Zb$
have been studied for the case of nonnegative cookies~\cite{Zer}, bounded number of   cookies per site~\cite{KZ08}, periodic cookies~\cite{KOS}, stationary-ergodic cookie distribution~\cite{ABO}, iterated leftover environments~\cite{AO},  site-based feedback environment~\cite{PT}, and Markovian cookie stacks~\cite{KP} (the last two papers are the most relevant to this paper).
Cookie random walks on $\Zb^d$ with $d>1$ are not studied as well as those on $\Zb$. Works which consider $d>1$ include 
 the case of a single cookie per vertex~\cite{BW},   the case of walks with a positive drift on one direction~\cite{Zer06} (note that $\Hor$--$\Ver$ walks have no drift), and a generalized case called \emph{generalized excited random walks}~\cite{MPRV}. 
We refer the reader to \cite{KZ} for an excellent exposition on this subject.

\emph{Reinforced
random walks}  (introduced by \cite{CD}, see also \cite{Pem88}) 
have the walkers choose their next location with probability dependent to
the number of visits to that location so far.
Thus the next exit from a vertex depends on all the past visits, instead of only the most recent visit.
  We refer the reader to \cite{Tot01,Pem07} for slightly dated but very useful surveys on this subject.
   Important works on reinforced random walks published after those surveys include \cite{ACK, ST,DST, SZ}.
  




 \medskip
 
 \subsection{Outline} In Section~\ref{secpreliminaries}, we introduce stack walks and random walks with local memory, and we present  zero-one laws for the recurrence of these walks.
 In Section~\ref{sechv walks}, we restate the main results in the notation of Section~\ref{secpreliminaries}.
 In Section~\ref{secmartingale}, we introduce the martingales and prove lower bounds for the return probabilities. 
 In Section~\ref{secproof of recurrnce many walkers}, we prove Theorem~\ref{thmrecurrence many walker}.
 In Section~\ref{secproof of recurrence IID subcritical}, we prove Theorem~\ref{thmrecurrence IID} for  $q\in (\frac{1}{3},1)$.
 In Section~\ref{secproof of recurrence IID critical}, we prove Theorem~\ref{thmrecurrence IID} for  $q=1$.
 In Section~\ref{secconcluding remarks}, we give concluding remarks and discuss open questions.

\bigskip

\section{Preliminaries}\label{secpreliminaries}

In this paper $G:=(V,E)$ is a (possibly infinite) graph that is simple, connected,
and locally finite (every vertex has finite degree).
A \emph{neighbor} of a vertex $x \in V$ is another vertex $y \in V$
such that $\{x,y\}$ is an edge of $G$.
We denote by $\Ngh(x)$ the set of neighbors of $x$. 
We denote by $\Nb$ the set of all nonnegative integers $\{0,1,2,\ldots\}$.

\medskip

\subsection{Stack walks}
	A \emph{stack} of $G$ is a function $\xi:V \times \Nb\to V$ such that, 
	for all $x \in V$ and $m \geq 0$, 
	we have  $\xi(x,m)$ is a neighbor of $x$.

\smallskip

\begin{definition}[Stack walks]
	Let $x$ be a vertex of $G$,  and let $\xi$ be a stack of $G$.
	A \emph{stack walk} with initialization $(x,\xi)$ is the sequence $(X_t,\xi_t)_{t \geq 0}$  defined recursively by
	\begin{equation}\label{equation: stack walk}
	\begin{split}
	(X_0,\xi_0)& \ := \ (x,\xi);\\
	\xi_{t+1}(x,\cdot)& \ := \ \begin{cases}
	\xi_{t}(x,\cdot+1) & \text{if $x= X_t$};\\
	\xi_{t}(x,\cdot) & \text{if $x\neq X_t$};
	\end{cases}\\
	X_{t+1}& \ := \ \xi_{t+1}(X_t,0).
	\end{split}
	\end{equation}
\end{definition}
Note that the stack walk is determined by the initialization $(x,\xi)$.

\smallskip

	The following image is useful: 
	$X_t$ records  the location of the walker at the end of the $t$-th step of the walk.
	For each $x \in V$,  we think of  $\xi_t(x,\cdot)$ as a stack lying under the vertex $x$ with $\xi_t(x,0)$ being on top, then $\xi_t(x,1)$, and so forth.
	Then, at the $t+1$-th step, the walker \emph{pops off} the stack of its current location $X_t$, meaning that it removes the top item of the stack
	lying under $X_t$.
	Then, the walker moves to the vertex specified by the top item of the new stack.
	This description of the walk  originated from
	the work of Diaconis and Fulton~\cite{DF},
	and the term \emph{stack walk} was coined by Holroyd and Propp~\cite{HP}.
	The stack walk was featured prominently in Wilson's
	algorithm for the uniform sampling of  spanning trees of a finite graph~\cite{Wil}.
	
\medskip

We will also consider the  multi-walker version of stack walks in this paper.

\smallskip

\begin{definition}[Multi-walker stack walks]
Let $n \geq 1$ be the number of walkers,
let $\xb = (x^{(1)}, \ldots, x^{(n)}) \in V^n$ be the initial location of the $n$ walkers,
and let $\xi$ be the  initial stack.
A \emph{turn order} is a sequence \. $\Oc = o_1o_2\ldots$ \.  such that each $o_t$ is contained in \. $\{0,\ldots,n\}$\. 
The \emph{multi-walker stack walk} \. $(\Xb_t=(X_t^{(1)},\ldots, X_t^{(n)}),\xi_t)_{t \geq 0}$ \.  is defined  recursively as follows:
\begin{equation*}
\begin{split}
(\Xb_0,\xi_0)& \ := \ (\xb,\xi);\\
\xi_{t+1}(x,\cdot)& \ := \ \begin{cases}
\xi_{t}(x,\cdot+1) & \text{if \ $x= X_{t}^{(i)}$ \ and \  $i=o_{t+1}$};\\
\xi_{t}(x,\cdot) & \text{otherwise};
\end{cases}\\
X_{t+1}^{(i)}& \ := \ 
\begin{cases}
\xi_{t+1}(X_t^{(i)},0) & \text{if \ $i=o_{t+1}$};\\
X_t^{(i)} & \text{otherwise.} 
\end{cases} 
\end{split}
\end{equation*}
\end{definition}
Note that, once the turn order $\mathcal{O}$ is fixed, the multi-walker stack walk is determined by the initialization $(\xb,\xi)$. 

\smallskip

Described in words, 
at the $t$-th step of the walk,
the $o_t$-th walker   performs one step of the stack walk, with $\Xb_t$ and $\xi_t$ tracking the location of the $n$ walkers and the stack after the first $t$ steps of the walk.
Note that the $n$ walkers are sharing the same stack when 
performing the stack walk.
Also note that $o_t=0$ means that none of the walkers move (for example, this can occur when some stopping condition has been reached, see \S\ref{subsecfrozen walks}).

\medskip

A multi-walker stack walk is \emph{recurrent} if every vertex is visited infinitely many times,
and is \emph{transient} if every vertex is visited at most finitely many times.
A stack $\xi$ of $G$ is \emph{regular} if, for every $x \in V$, every neighbor $y$ of $x$ appears in the stack $\xi(x,\cdot)$ infinitely many times.
A turn order is \emph{regular} if every walker performs infinitely many steps,
i.e., 
each $i \in \{1,\ldots,n\}$ appears in $\Oc$ infinitely many times.

\smallskip

\begin{lemma}[{\cite[Lemma~6]{HP}}]\label{lemma: dichotomy recurrence transience}
	Let \. $(\Xb_t,\xi_t)_{t \geq 0}$ \. be a multi-walker stack walk 
	with a regular turn order and a regular initial stack.
	Then the  walk is either recurrent or transient. \qed
\end{lemma}

\smallskip

\begin{proof}
	Suppose that the stack walk is not transient. 
	Then there exists $x \in V$ that is visited infinitely many times.
	It thus suffices to show that 
	\begin{equation}\label{eqrecurrence}
		\text{If a vertex $x \in V$ is visited infinitely many times, then the stack walk is recurrent.}
	\end{equation}
	Since the initial stack and the turn order are both regular, this implies that every neighbor of $x$ is also visited infinitely many times.
	Since $G$ is connected, 
	repeating  the same argument ad infinitum  implies that every vertex of $G$ is visited infinitely many times.
	This implies that the stack walk is recurrent if it is not transient, as desired.
\end{proof}

\smallskip

For every $x \in V$ and $t\geq 0$, we denote by \. $R_t(x) := R_t(x;\Xb,\xi,\Oc)$ \. the total number of transitions (also the number of visits) to $x$ by the stack walk  up to time $t$,
\begin{equation*}
 R_t(x) \ := \  \big|\{ s \in \{1,\ldots,t\} \. \mid \. X_{s-1}^{(i)}\neq x, \ X_{s}^{(i)}= x  \ \text{ for some } i \in \{1,\ldots,n\} \} \big|\.. 
\end{equation*}
We write \. $R_{\infty}(x) := \lim_{t \to \infty} R_{t}(x)$ \. the total number of transitions to $x$ throughout the entirety of the stack walk.

\smallskip

\begin{lemma}\label{lemmonotonicity}
	Let $\Oc$ and $\Oc'$ be  turn orders for two stack walks with the same initial location  and the same initial stack.
	Suppose that $\Oc'$ is a regular turn order.
	Then, for every  $x \in V$,
	\[  R_{\infty} (x) \quad \leq \quad R_{\infty}'(x),  \]
	 the total number of visits to $x$ by the stack walk with turn order $\Oc$ is less than or equal to  that of the stack walk with turn order $\Oc'$.
\end{lemma}

\smallskip

We present the proof of Lemma~\ref{lemmonotonicity} in Appendix~\ref{secappendixA},
and the proof is adapted  from \cite[Lemma~4.3]{BL} for abelian networks (see also \cite[Corollary~4.3]{CL} for an alternate proof).

\smallskip

\begin{lemma}[Abelian property]\label{lemabelian property}
		Let $\Oc$ and $\Oc'$ be  regular turn orders for two stack walks with the same initial location  and the same initial stack.
		Then the  stack walk with turn order  $\Oc$ is recurrent if and only if the stack walk with turn order $\Oc'$ is recurrent.
\end{lemma}

\smallskip

\begin{proof}
	This  follows  directly from Lemma~\ref{lemmonotonicity}.
\end{proof}

\smallskip

Thus we omit the dependence on the turn order from the notation, and 
we  assume  throughout this paper   the \emph{cyclic turn order} (i.e., $o_t$ is the unique integer in $\{1,\ldots,n\}$ that is equal to $t$ modulo $n$)  \. is used, unless stated otherwise.
Note that the cyclic turn order is regular.


\smallskip

For every $\xb \in V^n$ and every stack $\xi$, 
we say $(\xb,\xi)$ is \emph{recurrent}
if the stack walk with initialization $(\xb,\xi)$ is recurrent.
Similarly, we say 
$(\xb,\xi)$ is \emph{transient}
if the stack walk with initialization $(\xb,\xi)$ is transient.
We also write
\[   \satu_{\rec}(\xb,\xi)  \ := \ 
\begin{cases}
1 & \text{ if } (\xb,\xi) \text{ is recurrent};\\
0 & \text{ if } (\xb,\xi) \text{ is transient}.
\end{cases}\]



For every $x \in V$, the \emph{popping operation} $\varphi_x$ is a map  that sends a stack $\xi$ to the stack $\varphi_x(\xi)$  by popping off the top item of the stack of $x$,
\[  \varphi_x(\xi)(y,m) \ := \  \begin{cases}
\xi(y,m+1) & \text{ if } y = x;\\
\xi(y,m) & \text{ if } y \neq x.
\end{cases}  \]
We say that a stack \emph{$\xi'$ is obtained from $\xi$ by finitely many popping operations} 
if there exists $x_1,\ldots, x_m \in V$ 
such that \. $\varphi_{x_m} \circ \cdots \circ \varphi_{x_1}(\xi) \. = \. \xi'$\..
In particular, for every  stack walk \. $(\Xb_t,\xi_t)_{t \geq 0}$ \.  and every $t \geq 0$, the stack   
$\xi_{t}$ is obtained from $\xi_{0}$ by finitely many popping operations.

\smallskip

\begin{theorem}[c.f. {\cite[Theorem~1, Corollary~6]{AH2}}]\label{theorem: recurrence is invariant under stack relation}
	Let $\xb,\xb' \in V^n$ and let $\xi,\xi'$ be regular stacks such that $\xi'$ is obtained from $\xi$ by finitely many popping operations.
	Then $(\xb,\xi)$ is recurrent if and only if $(\xb',\xi')$
	is recurrent. \qed
\end{theorem}

\smallskip

We present the proof of Theorem~\ref{theorem: recurrence is invariant under stack relation} in Appendix~\ref{secappendixB} for completeness,
and the proof  is adapted  from \cite{AH2}.


\smallskip

\medskip

\subsection{Random walks with local memory}\label{subsecRWLM}

A random walk with local memory is the following randomized version of stack walks.
A \emph{rotor configuration} is a function $\rho:V \to V$ such that $\rho(x)$ is a neighbor of $x$  for all $x \in V$.
We assign a  \emph{local mechanism} $\MC_x$ for each $x \in V$, which  is a  Markov chain with states  $\Ngh(x)$ and with transition function \ts $p_x(\cdot,\cdot)$\ts. 
We assume that $\MC_x$ is an irreducible Markov chain.
Note that each $\MC_x$ is a Markov chain with a finite state space since the graph is locally finite.

\smallskip

\begin{definition}[Random walk with local memory]\label{definition: random walk with local memory}
	Let $n\geq 1$, let $\xb \in V^n$, and let $\rho$ be a rotor configuration.
	A  \emph{(multi-walker) random walk with local memory} with initialization $(\xb,\rho)$, or RWLM for short, is a random sequence of pairs  
	$(\Xb_t,\rho_t)_{t \geq 0}$   that satisfies the following transition rule:
	\begin{equation}\label{equation: transition rule RWLM}
	\begin{split}
	(\Xb_0,\rho_0) \ := \ & (\xb,\rho);\\
	\rho_{t+1}(x) \ := \ &\begin{cases}  Y_{t} & \text{if \ $x= X_{t}^{(i)}$ \ and \  $i\equiv t+1$ mod $n$};\\
	\rho_{t}(x) & \text{otherwise,} \end{cases}; \text{ and }\\
	X_{t+1}^{(i)} \ := \ &
	\begin{cases}
		Y_{t}  & \text{if  \  $i\equiv t+1$ mod $n$};\\
		X_t^{(i)} & \text{otherwise,}
	\end{cases}
	\end{split}
	\end{equation}
	where $Y_{t}$ is a random neighbor of $X_t^{(i)}$ sampled from \. $p_{X_t^{(i)}}(\rho_t(X_t^{(i)}),\cdot)$ \.  with $i$ satisfying  $i\equiv t+1$ mod $n$,
	independent of the past.
\end{definition}
We refer to \cite{CGLL} for history and references for random walks with local memory.

\smallskip

The following image is useful:
The random walk with local memory \. $(\Xb_t,\rho_t)_{n\geq 0}$ \. corresponds to the (random)  stack walk \. $(\Xb_t,\xi_t)_{t\geq 0}$\.,
where, for each $x\in V$,  the initial stack \. $\xi_0(x,\cdot)$ \.   is a  Markov process for the  Markov chain $\MC_x$ with initial state $\rho(x)$. 
We denote by  $\Stack(\rho)$ the corresponding probability distribution on  stacks of $G$, which is the source of the randomness in the RWLM. 
Note that, for each $t \geq 0$,
the vector $\Xb_t$ records the location of the walkers for both walks (the RWLM \emph{and}  the (random) stack walk) and  the rotor  $\rho_t(x)$ at $x$ corresponds to the  top item $\xi_t(x,0)$ of   the stack $\xi_t$.


\medskip

Let $\rho$ be an arbitrary rotor configuration of $G$.
The \emph{recurrence probability} \. $p_{\rec}(\rho)$ \. is the probability that the RWLM with initial rotor configuration $\rho$ is recurrent, 
\[ p_{\rec}(\rho) \ := \  \Eb_{\xi\sim \Stack(\rho)} \. \big[ \satu_{\rec}(\xb,\xi) \big]\.,  \]
where $\xb$ is an arbitrary element of $V^n$.
Note that \. $p_{\rec}(\rho)$ \. does not depend on the choice of $\xb \in V^n$ by Theorem~\ref{theorem: recurrence is invariant under stack relation}.
However, \. $p_{\rec}(\rho)$ \. could depend on the number of walkers $n$ (the dependence on $n$ is  not reflected in the notation for simplicity).

We say that two rotor configuration $\rho,\rho'$ \emph{differ by finitely many vertices} if there exists $x_1,\ldots, x_m$ such that \. $\rho(x)  = \rho'(x)$ \. for all $x \in V \setminus \{x_1,\ldots, x_m\}$\..

\smallskip

\begin{lemma}\label{lemma: recurrence rho invariant under finite changes}
	Let $\rho$ and $\rho'$ be rotor configurations that differ by finitely many vertices.
	Then 
	\[ p_{\rec}(\rho) \ = \ p_{\rec}(\rho')\.. \]
\end{lemma}

\smallskip

\begin{proof}
	Let $\xi$ be the stack sampled from $\Stack(\rho)$.
	Note that $\xi$ is a regular stack a.s. since every local mechanism $\MC_x$ is an irreducible Markov chain with a finite state space.
	For each $y \in V$, let $T_y$ be the (random) smallest nonnegative integer such that \. $\xi(y,T_y) \ = \ \rho'(y)$\..
	We write \. $\varphi := \prod_{y \in V} \varphi_y^{T(y)}$\..
	Then, by the Markov property, 
	\begin{equation}\label{equation: alfa 1}
	 \varphi\big(\xi\big)   \ \text{ is equal in distribution to the stack sampled from }\. \Stack(\rho')\..  
	\end{equation}
	
	Note that $T_y=0$ for all but finitely many $y$'s since $\rho$ and $\rho'$ differ by finitely many vertices.
	Also note that each $T_y$ is finite a.s. since the Markov chain $\MC_y$ is irreducible and finite.
	These two observations imply that  
	$\varphi$ is a finite product of popping operations a.s..
	It then follows from Theorem~\ref{theorem: recurrence is invariant under stack relation} that, for all $\xb \in V^n$,
	\begin{align}\label{equation: alfa 2}
		 \satu_{\rec} \big(\xb,\varphi \big( \xi \big)\big)   \ = \ \satu_{\rec} \big(\xb,\xi\big) \quad \text{a.s.} 
	\end{align}
	
	Combining \eqref{equation: alfa 1} and \eqref{equation: alfa 2}, we get 
	\begin{align*}
		 p_{\rec}(\rho) \quad & = \quad   \Eb_{\xi\sim \Stack(\rho)} \. \big[ \satu_{\rec}(\xb,\xi) \big] \quad =_{\eqref{equation: alfa 2}} \quad  \Eb_{\xi\sim \Stack(\rho)} \. \big[ \satu_{\rec}\big(\xb,\varphi(\xi)\big) \big]\\
		\quad &  =_{\eqref{equation: alfa 1}} \quad  \Eb_{\xi'\sim \Stack(\rho')} \. \big[ \satu_{\rec}(\xb,\xi') \big] \quad = \quad p_{\rec}(\rho')\.,
	\end{align*}
	as desired.
\end{proof}

\medskip

\subsection{Zero-one laws for recurrence}\label{subseczero-one laws}

In this subsection we prove several  zero-one laws for the recurrence of RWLMs.
(These zero-one laws are not to be confused with zero-one laws for the directional transience of cookie random walks~(see e.g., \cite{KZ})).


\smallskip

\begin{proposition}\label{proposition: zero-one law recurrence}
	Consider an RWLM with the initial rotor configuration  $\rho$.
	Then the recurrence probability satisfies
	\[ p_{\rec}(\rho) \ \in \ \{0,1\}\..  \]
\end{proposition}

\smallskip
We then say that a rotor configuration $\rho$ is \emph{recurrent} (with respect to the RWLM) if $p_{\rec}(\rho)=1$, and is \emph{transient} if $p_{\rec}(\rho)=0$.

%
%
%

We now present the proof of Proposition~\ref{proposition: zero-one law recurrence}.

\begin{proof}[Proof of Proposition~\ref{proposition: zero-one law recurrence}]
	Let $\xb$ be an arbitrary element of $V^n$,
	 let  $(\Xb_t,\rho_t)_{t\geq 0}$ be the RWLM with initialization $(\xb,\rho)$,
	 and let  $(\Xb         _t,\xi_t)_{t\geq 0}$ be the corresponding stack walk.
	 We write \. $\Fc_m:=\sigma(\Xb_0,\rho_0,\ldots \Xb_m,\rho_m)$ \. ($m\geq 0$)\..
	Note that $(\Fc_{m})_{m \geq 0}$ is a filtration by definition.
	
	
	Let $m \geq 0$. 
	Note that 
	the stack walk \. $(\Xb_t,\xi_t)_{t\geq 0}$ \. is recurrent  \  if and only if  \  the stack walk \. $(\Xb_{t+m},\xi_{t+m})_{t\geq 0}$ \. is recurrent. 
	This implies that,  the pair \. $(\xb,\xi)$ \. is recurrent,  if and only if,  the  pair  \. $(\Xb_m,\xi_m)$ \. is recurrent.
	It then follows that 
	\begin{align*}
	\Eb_{\xi \sim \Stack(\rho) } \big[ \satu_{\rec}\big(\xb, \xi\big) \. \big | \.  \Fc_m \big] \quad & =  \quad   \Eb_{\xi \sim \Stack(\rho) } \big[ \satu_{\rec}\big(\Xb_{m}, \xi_{m}\big) \. \big | \.  \Fc_m \big] \\
	\quad & =_{\text{Thm}~\ref{theorem: recurrence is invariant under stack relation}}   \quad \Eb_{\xi \sim \Stack(\rho) } \big[ \satu_{\rec}\big(\xb, \xi_{m}\big) \. \big | \.  \Fc_m \big]\..
	\end{align*}
	Now note that  $\xi_m$ is equal in distribution to $\Stack(\rho_m)$,
	and  $\rho_m$  is adapted to $\Fc_m$.
	Plugging these observations into the equation above,
		\begin{align}\label{equation: alfa 3}
	\Eb_{\xi \sim \Stack(\rho) } \big[ \satu_{\rec}\big(\xb, \xi\big) \. \big | \.  \Fc_m \big] 
	\quad = &  \quad \Eb_{\xi' \sim \Stack(\rho_m) } \big[ \satu_{\rec}\big(\xb, \xi'\big)  \big] \quad = \quad p_{\rec}(\rho_m)\..
	\end{align}
	(Note that  $p_{\rec}(\rho_m)$ here is regarded as  a random variable adapted to $\Fc_m$.)
	On the other hand, since $\rho_m$ differs from $\rho$ by at most $m$ many vertices,
	it follows from Lemma~\ref{lemma: recurrence rho invariant under finite changes} that
	 \ 	$p_{\rec}(\rho_m) \. = \. p_{\rec}(\rho)$\..
	 Combining this observation with \eqref{equation: alfa 3} gives us  
	 \begin{align}\label{equation: alfa 4}
	 	\Eb_{\xi \sim \Stack(\rho) } \big[ \satu_{\rec}\big(\xb, \xi\big) \. \big | \.  \Fc_m \big] 
	 	\quad = &  \quad p_{\rec}(\rho).
	 \end{align}

	Thus we have
	\begin{equation}\label{equation: zero-one law 1}
	p_{\rec}(\rho)  \quad =  \quad  \Eb_{\xi \sim \Stack(\rho) } \big[ \satu_{\rec}\big(\xb, \xi\big) \. \big | \.  \Fc_m \big]  \quad  \overset{m \to \infty}{\longrightarrow}  \quad \satu_{\rec}\big(\xb, \xi\big)  \qquad \Fc_{\infty}\text{-a.s.},
	\end{equation}
where the convergence in the equation above is due to Levy's zero-one law~(see e.g.~{\cite[Theorem~4.6.9]{Dur}}).
	Since $p_{\rec}(\rho)$ is a constant that does not depend on $m$, it then follows from \eqref{equation: zero-one law 1} that $p_{\rec}(\rho)\in \{0,1\}$, as desired.
\end{proof}

\medskip

Let \. $\IUD$ \.  be the probability 
measure on rotor configurations of $G$ where the rotor for $x \in V$ is  chosen independently and uniformly at random from the neighbors of $x$.
(By an abuse of the notation, we refer to this measure as the $\IUD$ measure, even though two different vertices can have different sets of neighbors, and thus different rotor distributions.)
We now consider RWLMs in which the initial rotor configuration is sampled from 
\. $\IUD$\.. 
Note that there are two (independent) layers of randomness involved in this process now:
The first layer being the random initial rotor configuration,
and the second layer being the transition rules of RWLM in \eqref{equation: transition rule RWLM}.


\smallskip

\begin{proposition}\label{proposition: zero-one law IUD}
	Consider an RWLM with the initial rotor configuration $\rho$ sampled from the $\IUD$ measure.
	Then exactly one of the following scenarios holds:
	\begin{itemize}
		\item $p_{\rec}(\rho) \ = \ 0$  \  for almost every  $\rho$ sampled from $\IUD$; or  
		\item $p_{\rec}(\rho) \ = \ 1$  \ for almost every  $\rho$ sampled from $\IUD$.  
	\end{itemize}
\end{proposition}

\smallskip

%
%
%


\smallskip

\begin{proof}
		Let $x_1,x_2,\ldots $ be an arbitrary ordering of elements of  $V$. Let  \. $\Fc'_{m}:=\sigma(\rho(x_m), \rho(x_{m+1}),\ldots )$ \. be the $\sigma$-field on rotor configurations of $G$ that depends only on the rotor configuration $\rho$ at $V \setminus \{x_1,\ldots, x_{m-1}\}$.
	Let $A$ be the set of rotor configurations of $G$ given by
	\[ A \ := \ \{ \rho \mid  p_{\rec}(\rho)=1 \}. \]
	By Theorem~\ref{theorem: recurrence is invariant under stack relation},
	for arbitrary rotor configurations $\rho,\rho'$ that differ by finitely many vertices, 
	$\rho$ is contained in $A$ implies $\rho'$ is also contained in $A$.
	This implies that $A \in \Fc'_{m}$ for every $m \geq 0$,
	and hence $A$ is contained in the  tail $\sigma$-field \. $\bigcap_{m \geq 0} \Fc'_m$.
	Since $\rho$ is sampled from the $\IUD$ measure,
	it then follows from Kolmogorov's zero-one law~(see e.g., \cite[Thm~2.5.3]{Dur})
	that
%
	 \[ \Pb_{\rho \sim \IUD} [ A ] \ \in \ \{0,1\}.  \]
	On the other hand, Proposition~\ref{proposition: zero-one law recurrence} gives us  
	\begin{align*}
		\Pb_{\rho \sim \IUD} [ A ] \ = \ 0  \qquad \text{implies} \qquad
		 \text{$p_{\rec}(\rho) \ = \ 0$  \ for almost every  $\rho$ sampled from $\IUD$;}\\
		 \Pb_{\rho \sim \IUD} [ A ] \ = \ 1  \qquad \text{implies} \qquad
		 \text{$p_{\rec}(\rho) \ = \ 1$  \ for almost every  $\rho$ sampled from $\IUD$\..}
	\end{align*}
	The proposition now follows from combining the two observations above.
\end{proof}

\bigskip

\section{Horizontal-vertical walks}\label{sechv walks}

The \emph{horizontal-vertical walk}, or \emph{$\Hor$--$\Ver$ walk} for short,  (introduced by Chan et al.~\cite{CGLL}),
is a nearest-neighbor random walk on $\Zb^2$,
where each  vertex of $\Zb^2$ is labeled either $\Hor$ or $\Ver$.
Initially we have  $n$ walkers dropped to  the origin $(0,0)$ in $\Zb^2$.
Each of the $n$ walkers performs the following move in a cyclic order:
the chosen walker
 changes the label of its current location with probability $q\in [0,1]$, and does not change the label with probability $1-q$.
Then, the walker takes a mean zero horizontal step if the new label is $\Hor$,
and a mean zero vertical step if the new label is $\Ver$. 

\smallskip

Formally, the  $\Hor$--$\Ver$ walk is an instance of RWLM (see~\eqref{equation: transition rule RWLM}) where each local mechanism $\MC_x$ is given by the transition rule 
\begin{itemize}
	\item If \. $y_1-x \in \{(1,0),-(1,0)\}$\., then:
	\[ p_x(y_1,y_2) \ = \ 
		\begin{cases}
		\frac{1-q}{2} & \text{ if }\  y_2-x \in \{(1,0),-(1,0)\}; \\
		\frac{q}{2} & \text{ if }\  y_2-x \in \{(0,1),-(0,1)\}.
		\end{cases}   \] 
	\item If \. $y_1-x \in \{(0,1),-(0,1)\}$\., then:
	\[ p_x(y_1,y_2) \ = \ 
	\begin{cases}
	\frac{q}{2} & \text{ if }\  y_2-x \in \{(1,0),-(1,0)\}; \\
	\frac{1-q}{2} & \text{ if }\  y_2-x \in \{(0,1),-(0,1)\}.
	\end{cases}   \] 
\end{itemize}
Note that this transition rule is equal to the one described in the beginning of the section.
Indeed, the correspondence is 
\begin{itemize}
	\item A vertex $x$  is labeled $\Hor$ if  \. $\rho(x)-x \in \{(1,0), -(1,0)\}$\.; and 
	\item A vertex $x$  is labeled $\Ver$ if  \. $\rho(x)-x \in \{(0,1), -(0,1)\}$\..
\end{itemize}

We will adopt the following notation throughout the rest of this paper:
\begin{itemize}
	\item A rotor configuration $\rho$ is simultaneously, a function $\rho: V \to V$, and a function $\rho: V \to \{\Hor,\Ver\}$, by the correspondence above;
	\item $q$ is strictly greater than $0$. This is  so that, for each $x \in V$,  the local mechanism $\MC_x$ for the $\Hor$--$\Ver$ walk is irreducible.
	\item  $\Pb$ \. and  \. $\Eb$ \. will be shorthands for $\Pb_{\xi \sim \Stack(\rho)}$\. and \. $\Eb_{\xi \sim \Stack(\rho)}$\., respectively.
	Recall from Section~\ref{subsecRWLM} that  $\xi$ is the random stack on $\Zb^2$ that arises from the Markov chains corresponding to the local mechanisms of the $\Hor$--$\Ver$ walk with initialization $\rho$.
	(Note that the random stack $\xi$ is the source of randomness for the $\Hor$--$\Ver$ walk.)
	
\end{itemize}


\medskip

We now restate the main results in Section~\ref{secintro} using the notation in Section~\ref{secpreliminaries}.

\smallskip

\begin{reptheorem}{thmrecurrence many walker}
	Let \. $q >0$ \.  and \. $n \geq \lfloor \frac{|4q-2|}{q} \rfloor + 1$\..
	Then, for every initial rotor configuration $\rho$, the corresponding $\Hor$--$\Ver$ walk with $n$ walkers satisfies
		\[p_{\rec}(\rho)=1.  \]
\end{reptheorem}

\smallskip

We will prove Theorem~\ref{thmrecurrence many walker} in Section~\ref{secproof of recurrnce many walkers}.
The following is a corollary of Theorem~\ref{thmrecurrence many walker} for walks with  a single walker.

\smallskip

\begin{repcorollary}{correcurrence single walker}
	Let $q \in (\frac{2}{5},\frac{2}{3})$ and $n=1$.
	Then, for every initial rotor configuration $\rho$, the corresponding $\Hor$--$\Ver$ walk with a single walker satisfies
	\[\pushQED{\qed} 
	 p_{\rec}(\rho)=1.  \qedhere 
	 \popQED\]
\end{repcorollary}

\smallskip

When $\rho$ is sampled from $\IUD$ measure, the recurrence regime in Corollary~\ref{correcurrence single walker} can be expanded.

\smallskip

\begin{reptheorem}{thmrecurrence IID}
	Let $q \in (\frac{1}{3},1]$, and let $n=1$.
	Then, for 
	the $\Hor$--$\Ver$ walk with a single walker,
	\[\text{$p_{\rec}(\rho) \ = \ 1$  \qquad  for \ almost every rotor configuration \. $\rho$ sampled from $\IUD$\..} \]
\end{reptheorem}

\smallskip

We split the proof of Theorem~\ref{thmrecurrence IID} into two parts:
we prove the case  \. $q \in (\frac{1}{3},1)$ \. in Section~\ref{secproof of recurrence IID subcritical}, and the case 
\. $q =1$ \. in Section~\ref{secproof of recurrence IID critical}.

\bigskip

\section{Martingale method}\label{secmartingale}

In this section we construct  a  martingale that tracks the number of departures from the origin by $\Hor$--$\Ver$ walks, and then we apply the optional stopping theorem   to get a lower bound for the return probabilities of $\Hor$--$\Ver$ walks.

\smallskip

\subsection{Frozen $\Hor$--$\Ver$ walks}\label{subsecfrozen walks}
Let $r \geq 1$.
We denote by \. $B_r$ \. the set of vertices in  $\Zb^2$ that are of Euclidean distance strictly less than $r$ from the origin, and by $\partial B_r$ the outer boundary of $B_r$,
\[B_r \ := \ \{ x \in \Zb^2 \. \mid \. |x| < r\}; \qquad \partial B_r \ := \ \{ x \in \Zb^2 \setminus B_r \. \mid \.  N(x) \cap B_r  \. \neq \.  \varnothing   \}\..\]

We now consider the variant of $\Hor$--$\Ver$ walks where each walker is immediately frozen if it reaches $\partial B_r$.

\begin{definition}[Frozen $\Hor$--$\Ver$ walks]
Let $r \geq 1$,
and let $\rho$ be a rotor configuration.	
	The \emph{frozen $\Hor$--$\Ver$ walk} 
	 \. $(\Yb_t:=(Y_t^{(1)},\ldots, Y_t^{(n)}),\zeta_t)_{t\geq 0}$ \. is defined recursively by
	 \begin{enumerate}
	 \item Initially $(\Yb_0,\zeta_0) \ := \  \big( \big((0,0),\ldots, (0,0)\big),\rho\big)$.
	 \item At the $t$-th step of the walk, let $i_{t}$ be the unique integer in $\{1,\ldots,n\}$ such that $i_{t} \equiv t \text{ mod } n$.
	 \item  Let the $i_{t}$-th walker perform one step of the $\Hor$--$\Ver$ walk if 
	 its current location \. $Y_{t}^{(i_{t})}$ \. is not contained in $\partial B_r$,
	 and let the walker skip its turn if its current location is contained in $\partial B_r$\..
	 \end{enumerate}	 
\end{definition}

Note that both $\Yb_t$ and $\zeta_t$  are random variables that depend on $r$, and  we do not write out their dependence on $r$ to lighten the notation.

\smallskip

We denote by  \. $R_{t,r}$ \.  ($t\geq 0$) \.  the number of returns to  the origin  up to  $t$ by the frozen walk,
\[ R_{t,r} \quad :=  \quad \big | \big \{ s \in \{1,\ldots, t\} \mid Y_{s-1}^{(i)}\neq (0,0), \  Y_{s}^{(i)} = (0,0), \ \text{ for some } i \in \{1,\ldots,n\} \big\} \big |\..  \]

The \emph{return probability} \. $p_{k,r}(\rho)$ \. is the probability that the walkers return to  the origin  at least  $k$ times during the lifetime of the frozen walk, 
\begin{equation}\label{eqreturn probability}
	p_{k,r}(\rho) \quad := \quad  \lim_{t \to \infty} \Pb \big[R_{t,r} \geq k \big]\..  
\end{equation}

 \medskip
 
 Recall that  \. $(\Xb_t,\rho_t)_{t\geq 0}$ \. is the \emph{unfrozen} $\Hor$--$\Ver$ walk (with the same initial rotor configuration $\rho$),
 and \. $p_{rec}(\rho)$ \. is the probability that  $(\Xb_t,\rho_t)_{t\geq 0}$ is recurrent (i.e., every vertex  is visited infinitely many times).
 
 \smallskip 
 \begin{lemma}\label{lemreturn probability limit}
 Let $q> 0$.
 Then, for the $\Hor$--$\Ver$ walk with the initial rotor configuration $\rho$,
 \[ p_{\rec}(\rho)   \ = \ \lim_{k\to \infty}  \lim_{r \to \infty} p_{k,r}(\rho)\.. \]
 \end{lemma}
 \smallskip
 
 \begin{proof}

 	
 	The event that $(\Xb_t,\rho_t)_{t\geq 0}$ is recurrent is equivalent to the event
 	\[ \bigcap_{k \geq 1} \, \bigcup_{t \geq 1} \,  \bigcup_{r \geq 1} \, \{ R_{t,r} \geq k\}. \]
 	We then have 
 	\begin{align*}
 		 & p_{\rec}(\rho)   \ = \  \Pb \bigg[\bigcap_{k \geq 1} \, \bigcup_{t \geq 1} \,  \bigcup_{r \geq 1} \, \{ R_{t,r} \geq k\}  \bigg]  \ = \  \lim_{k \to \infty} \Pb \bigg[ \bigcup_{t \geq 1} \,  \bigcup_{r \geq 1} \, \{ R_{t,r} \geq k\}  \bigg] \\
 		  \quad &= \  \lim_{k \to \infty} \Pb \bigg[ \bigcup_{r \geq 1} \,  \bigcup_{t \geq 1} \, \{ R_{t,r} \geq k\}  \bigg]  \ = \ \lim_{k\to \infty}  \lim_{r \to \infty} \lim_{t \to \infty} \Pb \big[R_{t,r} \geq k \big] \ = \
 		 \lim_{k\to \infty}  \lim_{r \to \infty} p_{k,r}(\rho)\.,
 	\end{align*}
 	where the second equality is because of the monotonicity in $k$, and the fourth equality is because of the monotonocity in $r$ and $t$.
 	This proves the lemma.
 \end{proof}

\medskip

\subsection{The martingale}
The \emph{potential kernel} $a:\Zb^2 \to \Rb$ for two-dimensional random walks is 
\[a(x): \quad = \quad \lim_{m \to \infty} \. \sum_{i=0}^{m-1} \. \big(p_i(0,0)- p_i(x)\big),   \]
where $p_i(x)$ denotes the probability for the simple random walk in $\Zb^2$ starting at the origin to visit $x$ at the $i$-th step.
Equivalently, the potential kernel is the unique nonnegative function of sublinear growth satisfying
\begin{align}\label{equation: potential kernel harmonic}
a(0,0) \ = \ 0, \qquad \text{ and } \qquad  
a(x) \ = \ -\mathbbm{1}\{x=o\} + \frac{1}{4} \sum_{y \in \Ngh(x)} a(y) \qquad (x\in \Zb^2)\.,
\end{align}
 by the uniqueness principle for harmonic functions on $\Zb^2$. 
We refer the reader to \cite{Law} for  references for the potential kernel. 

Let \. $\ft_{\Hor}: \Zb^2 \to \Rb$ \. and \. $\ft_{\Ver}: \Zb^2 \to \Rb$ \. be functions defined by 
\begin{align*}
	\ft_{\Hor}(x) \ :=& \   \frac{a\big(x+ (0,1)\big) \. + \. a\big(x  -(0,1)\big) \. -\.2a(x)}{4};\\
	\ft_{\Ver}(x) \ :=& \   \frac{a\big(x+ (1,0)\big) \. + \. a\big(x  -(1,0)\big) \. -\.2a(x)}{4}.
\end{align*}
Note that we have from \eqref{equation: potential kernel harmonic} that
\begin{equation}\label{equation: weight harmonic}
\ft_{\Hor}(x)+\ft_{\Ver}(x) \quad = \quad  \mathbbm{1}\{x=(0,0)  \}.   
\end{equation}


\smallskip

\begin{definition}\label{defmartingale}
	The martingale \. $M_t := M_t(r,\rho)$ \. ($t\geq 0$) for the frozen walk $(\Yb_t,\zeta_t)_{t \geq 0}$ is   
	\begin{equation*}
	M_t  \quad  := \quad   \sum_{i=1}^n a(Y_t^{(i)}) \  - \  N_t  \. +  \.  \frac{2q-1}{q}\sum_{x \in \bigcup_{i=1}^n \{Y_0^{(i)},\ldots, Y_{t}^{(i)}\}} \wt(x,t) \. - \. \wt(x,0),   
	\end{equation*}
	where  \. $N_t$ \. is the number of departure from the origin by the frozen walk up to time $t$,
	\[ N_{t} \quad :=  \quad \big | \big \{ s \in \{1,\ldots, t\} \mid Y_{s-1}^{(i)}= (0,0), \  Y_{s}^{(i)} \neq (0,0), \ \text{ for some } i \in \{1,\ldots,n\} \big\} \big |\.,  \]
	 and   \. $\wt(x,t)$ \.  is the \emph{weight} of the rotor of $x$ at time $t$ given by
	\begin{equation}\label{eqweight}
		\wt(x,t)  \ := \   
		\begin{cases}
			f_{\Hor}(x) & \text{ if } \  \zeta_t(x) = \Hor\.;\\
			f_{\Ver}(x) & \text{ if } \  \zeta_t(x) = \Ver\..
		\end{cases}
	\end{equation} 
\end{definition}

\smallskip

(Note that the martingale is not defined when $q=0$.)
Throughout  the rest of this paper,
 we denote by 
$\mathcal{F}_t:=\sigma(\Yb_0,\zeta_0,\ldots, \Yb_t,\zeta_t)$ \.  the $\sigma$-algebra generated by the  first $t$ steps of the frozen walk. 

\smallskip
\begin{lemma}
	The sequence $(M_{t})_{t \geq 0}$ is a martingale with respect to the filtration $(\Fc_t)_{t \geq 0}$.
\end{lemma}
\begin{proof}
	It follows immediately from the definition that \. $\Eb[M_0]  = 0$ \. and  \. $\Eb[|M_{t}|]<\infty$\..
	%
	Now note that, by the definition of $M_t$,
	\begin{equation}\label{equation: martingale 1}
	\begin{split}
	M_{t+1}-M_t \  =  \  a(Y_{t+1}^{(i)}) \. - \.  a(Y_t^{(i)}) \. - \. \mathbbm{1}\{ Y_t^{(i)}=(0,0) \} \. + \.  \frac{2q-1}{q} \big(\wt(Y_t^{(i)},t+1)  \. - \. \wt(Y_t^{(i)},t)\big),
	\end{split}
	\end{equation}
	where $i$ is the element in $\{1,\ldots,n\}$ such that $i \equiv t+1 \text{ mod } n$.
	 
	We now show that \ $\Eb[M_{t+1}\mid \Fc_t]=M_t$.
	We will restrict to the event   \. $\zeta_t(Y_t^{(i)})=\Hor$\., 
	as the proof for   the event \. $\zeta_t(Y_t^{(i)})=\Ver$ \. is identical.
	Since \. $\zeta_t(Y_t^{(i)})=\Hor$\., it follows from the dynamic of $\Hor$--$\Ver$ walks 
	that 
	\begin{equation*}
	\begin{split}
	\Eb[a(Y_{t+1}^{(i)}) \mid \Fc_t]
	 \ = & \   q\. \frac{a\big(Y_t^{(i)} +(0,1)\big)  \. + \. a\big(Y_t^{(i)} -(0,1)\big)}{2} \. + \. (1-q)\frac{a\big(Y_t^{(i)}+(1,0)\big) + a\big(Y_t^{(i)}-(1,0)\big)}{2}\\
	 \ = & \  a(Y_t^{(i)})  \. + \.  2q \ft_{\Hor}(Y_t^{(i)}) \. + \.  2(1-q) \ft_{\Ver}(Y_t^{(i)})\.,
	\end{split}
	\end{equation*}
	where the first equality is due to the transition rule of $\Hor$--$\Ver$ walks, and the second equality is due to the definition of $\ft_{\Hor}$ and $\ft_{\Ver}$.
	This is equivalent to 
	\begin{equation}\label{equation: martingale 2}
		\Eb[a(Y_{t+1}^{(i)})-a(Y_t^{(i)}) \mid \Fc_t]    \quad = \quad 2q \ft_{\Hor}(Y_t^{(i)}) \. + \.  2(1-q) \ft_{\Ver}(Y_t^{(i)})\.. 
	\end{equation}
	On the other hand, since  \. $\zeta_t(Y_t^{(i)})=\Hor$\.,
	\begin{equation}\label{equation: martingale 3}
	\begin{split}
	\Eb[\wt(Y_{t}^{(i)},t+1)-\wt(Y_{t},t) \mid \Fc_t]
	\quad =& \quad  q \ft_{\Ver}(Y_t^{(i)})  \. + \. (1-q)\ft_{\Hor}(Y_t^{(i)}) \ - \ \ft_{\Hor}(Y_t^{(i)}) \\ 
	\quad =& \quad  q \ft_{\Ver}(Y_t^{(i)})  \. - \. q\ft_{\Hor}(Y_t^{(i)})\..
	\end{split}
	\end{equation}
	where the first equality is due to the transition rule of $\Hor$--$\Ver$ walks.
	
	Plugging \eqref{equation: martingale 2} and  \eqref{equation: martingale 3} into \eqref{equation: martingale 1}, we get 
	\begin{align*}
		\Eb[M_{t+1}-M_t\mid \Fc_t] \quad & = \quad \ft_{\Hor}(Y_t^{(i)}) \. + \. \ft_{\Ver}(Y_t^{(i)}) \. - \. \mathbbm{1}\{ Y_t^{(i)}=(0,0) \}  
		\quad =_{\eqref{equation: weight harmonic}} \quad 0\.. 
	\end{align*}
	This completes the proof.
	%
	%
	%
	%
\end{proof}

\medskip

\subsection{Lower bound for the return probability}
Let $k\geq 0$.
We denote by \. $\tau_{k,r}(\rho)$ \. the first time (for the frozen $\Hor$--$\Ver$ walk)  to either, return to the origin  $k$ times, or,  have all walkers frozen at $\partial B_r$,
\begin{equation}\label{eqtau}
 \tau_{k,r}(\rho) := \min \. \{\.  t \geq 0 \. \mid \.   R_{t,r}=k \quad \text{ or } \quad  |Y_t^{(i)}|\in \partial B_r  \ \text{ for all } i \in \{1,\ldots,n\} \.  \}\., 
\end{equation}
Note that $\tau_{k,r}(\rho)$  is a stopping time  of the filtration $(\Fc_t)_{t\geq 0}$ by definition.
Also note that $\tau_{k,r}(\rho)<\infty$ a.s. by the following argument: 
Suppose  that the walker never reaches $\partial B_r$ (as otherwise we are done).
Then  some vertex  $x \in B_r$ is visited infinitely many times throughout the entirety of the walk.
This implies that every vertex in $B_r$ is visited infinitely many times a.s.. In particular, the origin is visited more than $k$ times a.s., which implies that $\tau_{k,r}(\rho)<\infty$ a.s..

The main result of this subsection is the following lemma.

\smallskip

\begin{lemma}\label{lemreturn probability lower bound}
	There exists \. $C >0$ \. such that the following inequality always hold:
	\begin{equation*}
		\begin{split}
		p_{k,r}(\rho) \quad \geq \quad  1  \ - \  \frac{\pi}{2\ln r}   \left( \frac{2q-1}{nq} \ts \sum_{x \in B_{r+1}}  \big( \wt(x,0)  \ - \  \Eb\big[\wt\big(x,\tau_{k,r}(\rho)\big)\big] \big)  \ + \  \frac{k}{n}  \ + \  C \right)\.\..
		\end{split}
	\end{equation*}
\end{lemma}

\smallskip

We now build toward the proof of Lemma~\ref{lemreturn probability lower bound}.
Our first ingredient is the following asymptotic estimate of the potential kernel.
For every $x \in \Zb^2$, the \emph{argument} $\arg(x)$ of $x$ is the unique real number in  \ts $(-\pi, \pi]$ \ts such that
\[  \frac{x}{|x|} \ = \ \big(\cos (\arg(x)), \sin(\arg(x))\big)\..  \]

\smallskip

\begin{theorem}[{\cite[Theorem~2]{FU}}]\label{thmpotential kernel} 
	For every $x \in \Zb^2\setminus \{(0,0)\}$,
	\[  a(x) \quad = \quad  \frac{2}{\pi}\ln |x| \ + \  \lambda \ -\ \frac{\cos(4\arg(x))}{6\pi|x|^2} \ + \   O({|x|^{-4}}),  \]
	where \. $\lambda \ := \ \frac{2}{\pi}\gamma+\frac{1}{\pi} \log 8$\.,
	with $\gamma$ being the Euler-Mascheroni constant. \qed
\end{theorem}

\smallskip

The second ingredient is the following version of the optional stopping theorem.

\smallskip

\begin{theorem}[Optional stopping theorem, see~{\cite[Theorem~10.10(ii)]{Wil91}}]\label{theorem: optional stopping}
	Let $(M_t)_{t\geq 0}$ be a martingale and let $\tau$ be a stopping time.
	If $(M_t)_{t\geq 0}$ is bounded and $\tau$ is a.s. finite,
	then $\Eb[M_{\tau}]= \Eb[M_0]$. \qed
\end{theorem} 

\smallskip

We now present the proof of Lemma~\ref{lemreturn probability lower bound}.

\smallskip

\begin{proof}[Proof of Lemma~\ref{lemreturn probability lower bound}]
	We fix $k \geq 0$ and the initial rotor configuration $\rho$ throughout this proof,
	and only $r$ is allowed to vary.
	We write \. $p_r:= p_{k,r}(\rho)$ \. and \. $\tau_r:=\tau_{k,r}(\rho)$\..
	
	Let $t \geq 0$, and let  \. $T_r:= t \wedge \tau_{r}$\..
	Note that, for the first $T_r$ steps of the walk,
	the walkers never leave $B_{r+1}$, and the total number of departures from the origin never exceeds $k+n$.
	Also note that $T_r$ is finite a.s. since $\tau_r$ is finite a.s..
	It then follows from definition of $M_t$ that, 
	\begin{align*}
		\big|M_{T_r}\big| \quad &\leq \quad   \sum_{i=1}^n \big|a(Y^{(i)}_{T_r})\big| \  + \  N_{T_r}  \ +  \  \frac{|2q-1|}{q}\sum_{x \in \bigcup_{i=1}^n \{Y_0^{(i)},\ldots, Y_{T_r}^{(i)}\}} |\wt(x,T_r)| \. + \. |\wt(x,0)|\\
		&\leq \quad   n \max_{x \in B_{r+1}}  |a(x) |   \  + \  k+n  \ +  \  \frac{|2q-1|}{q}\sum_{x \in B_{r+1} }  2 \max \{\big|\ft_{\Hor}(x)\big|,\big|\ft_{\Ver}(x)\big| \}\.,
	\end{align*}
	and note that the right side is bounded from above by a constant \. $c:=c(k,r) >0 $ \. that depends only on $k$ and $r$.
	Thus the conditions in Theorem~\ref{theorem: optional stopping} are satisfied,   and we have 
	\begin{align*}
	\Eb[M_{\tau_{r}}] \ = \ \lim_{t \to \infty}  \Eb[M_{T_{r}}] \ = \  \Eb[M_0] \ = \ 0.
	\end{align*}
	This implies that
	\begin{align}
	 \sum_{i=1}^n \Eb\big[a(Y^{(i)}_{\tau_{r}})\big] \quad &= \quad \Eb\big[N_{\tau_{r}}\big]  \ + \  \frac{2q-1}{q} \sum_{x  \in B_{r+1}} \big( \wt(x,0) \ - \ \Eb\big[\wt(x,\tau_r) \big]\big) \notag \\
	\quad & \leq  \quad k+n  \ + \ \frac{2q-1}{q}  \sum_{x  \in B_{r+1}} \big( \wt(x,0) \ - \ \Eb\big[\wt(x,\tau_r) \big]\big). \label{eqoptional stopping 1}
	\end{align}
	
	On the other hand, for every $i \in \{1,\ldots,n\}$,
	\begin{align*}
		\Eb\big[a(Y^{(i)}_{\tau_{r}})\big] \quad = \quad  \Pb\big[R_{\tau_r}  = k \big] \. \Eb\big[a(Y^{(i)}_{\tau_{r}}) \. \mid \. R_{\tau_r}  = k \big] \ + \ \Pb\big[R_{\tau_r}  < k \big] \. \Eb\big[a(Y^{(i)}_{\tau_{r}}) \. \mid \. R_{\tau_r}  < k \big]\..
	\end{align*}
	Now note that,  by definition of $p_r$ and $\tau_r$,
	\[ \Pb\big [R_{\tau_r}  = k \big] \ = \ p_r; \qquad \Pb\big [R_{\tau_r}  < k \big] \ = \ 1- p_r\.,  \]
	which gives us  
		\begin{align*}
	\Eb\big[a(Y^{(i)}_{\tau_{r}})\big] \quad = \quad  p_r \. \Eb\big[a(Y^{(i)}_{\tau_{r}}) \. \mid \. R_{\tau_r}  = k \big] \ + \ (1-p_r)  \. \Eb\big[a(Y^{(i)}_{\tau_{r}}) \. \mid \. R_{\tau_r}  < k \big]\..
	\end{align*}
	Now note that the potential kernel $a$ is a nonnegative function.
	Also note that, on the event \. $R_{\tau_r}  < k$\., 
	we have \. $Y^{(i)}_{\tau_{r}} \in \partial B_r$ \. for all $i \in \{1,\ldots,n\}$ by the definition of $\tau_r$.
	This implies that
	\begin{equation}\label{eqoptional stopping 2}
	\begin{split}
	\Eb\big[a(Y^{(i)}_{\tau_{r}})\big] \quad & \geq  \quad   (1-p_r)  \. \Eb\big[a(Y^{(i)}_{\tau_{r}}) \. \mid \. R_{\tau_r}  < k \big] \\
	\quad & \geq_{\text{Thm~\ref{thmpotential kernel}}} \quad (1-p_r) \. \frac{2 \ln r}{\pi}  \ - \ C'\.,
	\end{split}
	\end{equation}
	for some absolute constant $C'>0$.
	Plugging \eqref{eqoptional stopping 2} into \eqref{eqoptional stopping 1}, we get
	\begin{align*}
	(1-p_r) \. \frac{2 n\ln r}{\pi} \quad \leq \quad k +n \ + \  \frac{2q-1}{q} \sum_{x  \in B_{r+1}} \big( \wt(x,0) \ - \ \Eb\big[\wt(x,\tau_r) \big]\big)   \ + \ C'n \.,
	\end{align*}
	which is equivalent to 
	\[ p_r \quad \geq \quad 1 \ - \ \frac{k\pi}{2n\ln r}  \ - \ \frac{\pi}{2n\ln r} \ - \ \frac{\pi}{2n \ln r}\. \frac{2q-1}{q} \sum_{x  \in B_{r+1}} \big( \wt(x,0) \ - \ \Eb\big[\wt(x,\tau_r) \big]\big)   \ - \  \frac{C'\pi}{2\ln r}\..  \]
	The lemma now follows by taking 
	 $C:=C'+1$, as desired.
\end{proof}

\medskip

\subsection{Asymptotics of the weight function}
The following asymptotic estimates of the weight functions will be used in coming sections.

\smallskip

\begin{lemma}\label{lemweight function asymptotic}
	For every  $x \in \Zb^2 \setminus \{(0,0)\}$,
	\begin{align*}
	\ft_{\Hor}(x) \ = \ &  \frac{\cos(2\arg(x))}{2\pi|x|^2} \ + \ O(|x|^{-4});  \qquad 
		\ft_{\Ver}(x) \ = \   - \frac{\cos(2\arg(x))}{2\pi|x|^2} \ + \ O(|x|^{-4})\..
	\end{align*} 
\end{lemma} 

\smallskip

\begin{proof}
	We present only the proof of the asymptotics of $\ft_{\Ver}(x)$,
	as the proof for $\ft_{\Hor}(x)$ is identical.
	Let $z \in \mathbb{C}$ be the complex number \. $z:= b \. + \. c \ts i$\., where $(b,c):=x$.
	We then have 
	\begin{equation*}
	\begin{split}	
	& \ft_{\Ver}(x)  \  = \    \frac{a\big(x+(1,0)\big) \. + \. a\big(x-(1,0)\big)-2a(x)}{4}  \\
	\ & =_{\text{Thm~\ref{thmpotential kernel}}} \  \frac{1}{2\pi}\left( {\ln|z+1|+\ln|z-1|- 2\ln|z|}\right) \\
	 \ & \phantom{= - } \ - \  \frac{1}{24\pi}\left(\frac{\cos (4\arg(z+1))}{|z+1|^2}+ \frac{\cos (4\arg(z-1))}{|z-1|^2}-2 \frac{\cos (4\arg(z))}{|z|^2} \right)     \ + \ O(|x|^{-4}).
	 \end{split}
	\end{equation*} 

	Now note that 
	\begin{align*}
	 & {\ln|z+1|+\ln|z-1|- 2\ln|z|}  \quad = \quad   \ln \left|1-\frac{1}{z^2}\right| 
	 \quad = \quad \frac{1}{2} \ln \left(1-\frac{1}{z^2}\right)  \ + \ \frac{1}{2} \ln \left(1-\frac{1}{\overline{z}^2}\right)\\
	 & = \quad - \ \frac{1}{2z^2} \ - \ \frac{1}{2\overline{z}^2}  \ + \ O(|z|^{-4}) \quad = \quad - \. \frac{z^2+\overline{z}^2}{2|z|^4} \ + \ O(|z|^{-4}) 
	 \quad = \quad -\frac{\cos(2\arg(x))}{|x|^2} \ + \ O(|x|^{-4}).
	\end{align*}
	Also note that, by applying the mean value theorem to the function \. $g(z) =  \frac{\cos (4\arg(z))}{|z|^2}$\.,
	\[\frac{\cos (4\arg(z+1))}{|z+1|^2}+ \frac{\cos (4\arg(z-1))}{|z-1|^2}-2 \frac{\cos (4\arg(z))}{|z|^2}  \quad = \quad O(|z|^{-4}) \quad = \quad  O(|x|^{-4})\..  \]
	This lemma now follows.
\end{proof}

\smallskip

In particular, the following consequence of Lemma~\ref{lemweight function asymptotic} will be used in the coming sections.
We write 
\begin{equation}\label{eqweight max definition}
   \ft(x) \quad := \quad \max \. \{ \big|f_{\Hor}(x)\big|, \big|f_{\Ver}(x)\big|\}\..   
\end{equation}

It follows from Lemma~\ref{lemweight function asymptotic}
that, for every $x \in \Zb^2 \setminus \{(0,0)\}$, 
\begin{equation}\label{eqweight max asymptotic}
 |\ft(x)|  \quad \leq \quad \frac{|\cos(2\arg(x))|}{2\pi|x|^2} \. + \. O(|x|^{-4})\..
\end{equation}

\smallskip

\begin{lemma}\label{lemweight function sum}
	For every $r \geq 1$,
	\[ \sum_{x \in B_{r+1}}  |\ft(x)| \quad =  \quad  \frac{2}{\pi} \ln r \ + \ O(1)\..   \]
\end{lemma}

\smallskip

\begin{proof}
	It follows from Lemma~\ref{lemweight function asymptotic} that 
	\begin{align*}
		\sum_{x \in B_{r+1}}  |\ft(x)| \quad = \quad  \ft(0,0) \ + \ \sum_{x \in B_{r+1}\setminus \{(0,0)\}} \frac{|\cos(2\arg(x))|}{2\pi|x|^2} \. + \. O(|x|^{-4})\..
	\end{align*}
	By approximating the sum in the right side with the corresponding integral in $\Rb^2$,  we get 
		\begin{align*}
	& \sum_{x \in B_{r+1}}  |\ft(x)|  \quad = \quad  \int_{1 \leq |z| \leq r} \frac{|\cos(2\arg(z))|}{2\pi|z|^2} \, dz  \ + \ O(1)\\
	 & = \quad \frac{1}{2\pi} \. \int_{0}^{2\pi} \int_{1}^{r} \frac{|\cos(2\theta)|}{s} \, ds \,  d\theta \ + \  O(1) \quad = \quad \frac{2}{\pi} \ln r \ + \ O(1)\.,
	\end{align*}
	as desired.
\end{proof}
\bigskip

\section{Proof of Theorem~\ref{thmrecurrence many walker}}\label{secproof of recurrnce many walkers}

We now build toward the proof of Theorem~\ref{thmrecurrence many walker}\..
We start with the following lower bound for the return probability $p_{k,r}(\rho)$ that, most importantly, does not depend on  $k$ or $\rho$.
Recall the definition of return probability $p_{k,r}(\rho)$ from \eqref{eqreturn probability}.

\smallskip

\begin{lemma}\label{lemrecurrence many walker lower bound}
 For every $k \geq 0$ and every rotor configuration $\rho$,
 \[  \lim_{r \to \infty} p_{k,r}(\rho)  \quad \geq \quad  1 \ - \ \frac{|4q-2|}{nq}\.. \]
\end{lemma}

\smallskip

\begin{proof}
	We have from Lemma~\ref{lemreturn probability lower bound} that
	\begin{align*}
	\lim_{r \to \infty} p_{k,r}(\rho)  \quad & \geq \quad 1  \ - \  \limsup_{r \to \infty} \frac{\pi}{2 n \ln r} \ts \frac{2q-1}{q} \ts \sum_{x \in B_{r+1}}  \big( \wt(x,0)  \ - \  \Eb\big[\wt\big(x,\tau_{k,r}(\rho)\big)\big] \big)\..
	\end{align*}
	On the other hand, we have by definition of the weight function that
	\begin{equation}\label{eqweight upper bound crude 0}
	 \big| \wt(x,0)  \big|  \ , \  \Eb\big[\wt\big(x,\tau_{k,r}(\rho)\big)\big] \quad \leq \quad \max\{ \big |\ft_{\Hor}(x)\big|, \big|\ft_{\Ver}(x) \big|  \} \quad = \quad |\ft(x)|\..  	\end{equation}
	Plugging \eqref{eqweight upper bound crude 0} into the inequality above,
	\begin{align*}
	\lim_{r \to \infty} p_{k,r}(\rho)  & \geq \quad 1  \ - \  \limsup_{r \to \infty} \frac{\pi}{2 n \ln r} \ts \frac{|2q-1|}{q} \ts \sum_{x \in B_{r+1}}  2 |\ft(x)|\\
	& \geq_{\text{Lem~\ref{lemweight function sum}}} \quad 1  \ - \  \limsup_{r \to \infty} \frac{\pi}{2 n \ln r} \frac{|2q-1|}{q} \left( \frac{4 \ln r}{\pi} \ + \ O(1) \right)\\
	& \geq \quad 1 \ - \ \frac{|4q-2|}{nq}\.,
	\end{align*}
	as desired.
\end{proof}

\smallskip

We now present the proof of Theorem~\ref{thmrecurrence many walker}.

\smallskip
\begin{proof}[Proof of Theorem~\ref{thmrecurrence many walker}]
	We have 
	\begin{align*}
		p_{\rec}(\rho) \quad &=_{\text{Lem~\ref{lemreturn probability limit}}} \quad 
		\lim_{k \to \infty}\lim_{r \to \infty} p_{k,r}(\rho) 
		 \quad \geq_{\text{Lem~\ref{lemrecurrence many walker lower bound}}} \quad     1 \ - \frac{|4q-2|}{nq}\..
	\end{align*}
	Since \. $n > \frac{|4q-2|}{q}$ \. by assumption,
	it then follows from the equation above that 
	\[  	p_{\rec}(\rho)  \quad > \quad 0\..  \]
	Together with Proposition~\ref{proposition: zero-one law recurrence}, this implies that \. $p_{\rec}(\rho) \. = \. 1$\.,
	and the proof is complete.
\end{proof}

\bigskip

\section{Proof of Theorem~\ref{thmrecurrence IID}, case $q<1$}\label{secproof of recurrence IID subcritical}

In this section we prove Theorem~\ref{thmrecurrence IID} for  $q<1$.

\smallskip
\begin{proposition}\label{proprecurrence IID subcritical}
	Let $q \in (\frac{1}{3},1)$.
	Then, for  the $\Hor$--$\Ver$ walk with a single walker,
	\[\text{$p_{\rec}(\rho) \ = \ 1$  \qquad  for almost every  \. $\rho$ sampled from $\IUD$\..} \]
\end{proposition}

\smallskip

The proof of Proposition~\ref{proprecurrence IID subcritical} builds on the same argument used in the proof of Theorem~\ref{thmrecurrence many walker}.
The previous argument used the following trivial upper bound in \eqref{eqweight upper bound crude 0},
\[\big| \wt(x,0)  \big|    \quad \leq \quad \max\{ \big |\ft_{\Hor}(x)\big|, \big|\ft_{\Ver}(x) \big|  \}\..  \]
This upper bound can be improved by  the following lemma,
which  computes  the exact value of \. $ \Eb_{\rho \sim \IUD}\big[\wt(x,0)\big]$\.. (Note that this lemma applies to all values of $q$.)

\smallskip

\begin{lemma}\label{lemweight initial configuration}
	For all $r \geq 1$,
	\[  \sum_{x \in B_{r+1}}   \Eb_{\rho \sim \IUD}\big[\wt(x,0)\big] \quad = \quad \frac{1}{2}\..  \]
\end{lemma}

\smallskip

\begin{proof}
	We have from \eqref{equation: weight harmonic} that, for every $x \in \Zb^2$,
	\[  \Eb_{\rho \sim \IUD}\big[\wt(x,0)\big] \quad = \quad \frac{\ft_{\Hor}(x) \ + \ \ft_{\Ver}(x)}{2}  \quad = \quad   \frac{1}{2}\. \mathbbm{1}\{x=o  \}\., \]
	for which the lemma now follows.
\end{proof}

\smallskip

We now present the proof of Proposition~\ref{proprecurrence IID subcritical}.

\smallskip

\begin{proof}[Proof of Proposition~\ref{proprecurrence IID subcritical}]
	We have 
	\begin{equation}\label{eqdelta 1}
	\begin{split}
		\Eb_{\rho \sim \IUD} \big[p_{\rec}(\rho)  \big] \quad &=_{\text{Lem~\ref{lemreturn probability limit}}} \quad 
			\Eb_{\rho \sim \IUD} \big[ \lim_{k \to \infty}\lim_{r \to \infty} p_{k,r}(\rho)  \big] \\
			\quad  & = \quad \lim_{k \to \infty}\lim_{r \to \infty}  \Eb_{\rho \sim \IUD} \big[  p_{k,r}(\rho)  \big]\.,
	\end{split}
	\end{equation}
		where the second equality is due to the bounded convergence theorem.
	
	Now note that 
	\begin{equation*}
	\begin{split}
		& \lim_{r \to \infty} \Eb_{\rho \sim \IUD} \big[  p_{k,r}(\rho)  \big]\\
		& \geq_{\text{Lem~\ref{lemreturn probability lower bound}}} \quad 
		1  \ - \   \limsup_{r \to \infty} \frac{\pi}{2 n \ln r} \ts \frac{2q-1}{q} \ts \sum_{x \in B_{r+1}}  \big(  \Eb_{\rho \sim \IUD} \big[\wt(x,0)\big]  \ - \  \Eb_{\rho \sim \IUD} \. \Eb \big[\wt\big(x,\tau_{k,r}(\rho)\big)\big] \big)\\
		& =_{\text{Lem~\ref{lemweight initial configuration}}} \quad 
		1  \ - \   \limsup_{r \to \infty} \frac{\pi}{2 n \ln r} \ts \frac{2q-1}{q} \ts \sum_{x \in B_{r+1}}  \big(   - \.  \Eb_{\rho \sim \IUD} \. \Eb \big[\wt\big(x,\tau_{k,r}(\rho)\big)\big] \big)\..
	\end{split}
	\end{equation*}
	On the other hand, we have by the definition of the weight function that
	\begin{equation}\label{eqweight upper bound crude}
	 \big| \Eb \big[\wt\big(x,\tau_{k,r}(\rho)\big) \big] \big|  \quad \leq \quad  \max\{ \big|f_{\Hor}(x)\big|, \big|f_{\Ver}(x)\big|  \} \quad = \quad |\ft(x)| \..
	\end{equation}
	Plugging \eqref{eqweight upper bound crude} into the previous inequality, we get
	\begin{equation*}
	\begin{split}
	& \lim_{r \to \infty} \Eb_{\rho \sim \IUD} \big[  p_{k,r}(\rho)  \big] 	\quad  \geq \quad 
		1  \ - \   \limsup_{r \to \infty} \frac{\pi}{2 n \ln r} \ts \frac{|2q-1|}{q} \ts \sum_{x \in B_{r+1}}  |\ft(x)| \\ & =_{\text{Lem~\ref{lemweight function sum}}} \quad 
		1  \ - \   \limsup_{r \to \infty} \frac{\pi}{2 n \ln r} \ts \frac{|2q-1|}{q} \ts \left( \frac{2}{\pi} \ln r \ + \ O(1) \right) 
		\quad = \quad 1 \ -  \ \frac{|2q-1|}{nq}\..
	\end{split}
	\end{equation*}
	Plugging the inequality above into \eqref{eqdelta 1}, we get
	\begin{equation}\label{eqlower bound IID}
			 \Eb_{\rho \sim \IUD} \big[p_{\rec}(\rho)  \big]  \quad \geq  \quad 1 \ -  \ \frac{|2q-1|}{nq}\..
	\end{equation}
	Now note that the right side of \eqref{eqlower bound IID} is strictly greater than $0$ since $q \in (\frac{1}{3},1)$ and $n=1$\..
	 It now follows from Proposition~\ref{proposition: zero-one law IUD} that 
	 \. $p_{\rec}(\rho) \. = \. 1$  \. for almost every rotor configuration $\rho$ sampled from $\IUD$, as desired.
\end{proof}

\bigskip

\section{Proof of Theorem~\ref{thmrecurrence IID}, case $q=1$}\label{secproof of recurrence IID critical}

In this section we prove Theorem~\ref{thmrecurrence IID} for  $q=1$.

\smallskip
\begin{proposition}\label{proprecurrence IID critical}
	Let $q = 1$.
	Then, for  the $\Hor$--$\Ver$ walk with a single walker,
	\[\text{$p_{\rec}(\rho) \ = \ 1$  \qquad  for almost every  \. $\rho$ sampled from $\IUD$\..} \]
\end{proposition}

\smallskip

We now build toward the proof of Proposition~\ref{proprecurrence IID critical}.
Throughout this section,
 the $\Hor$--$\Ver$ walk has $n=1$ walker  and with $q=1$.
 With the exception of the proof of Proposition~\ref{proprecurrence IID critical}, the constants $k \geq 0$ and $r\geq 1$ are fixed.
We denote by  \. $(Y_t, \zeta_t)_{t \geq 0}$ \. the single-walker $\Hor$--$\Ver$ walk with the initial location $(0,0)$,
with the initial rotor configuration $\rho$,
and with the walker  frozen upon reaching $\partial B_r$.
Note that $Y_t$ is the location of the single walker, and $\zeta_t$ is the rotor configuration after the first $t$ steps of the walk.
Recall from \eqref{eqtau} that  \. $\tau(\rho):=\tau_{k,r}(\rho)$ \.  is the first time the walker either, returns to the origin $k$ times, or, reaches $\partial B_r$\..

The proof of Proposition~\ref{proprecurrence IID critical} follows almost the same argument as that of  Proposition~\ref{proprecurrence IID subcritical}.
Indeed, the previous  argument failed to give recurrence for the case $q=1$ because
we use the  trivial upper bound in \eqref{eqweight upper bound crude},
\[ \big| \Eb \big[\wt\big(x,\tau(\rho) \big) \big] \big|  \quad \leq \quad  |\ft(x)| \., \]
which in turn only  gives to the trivial lower bound  that the recurrence probability is nonnegative.
Thus Proposition~\ref{proprecurrence IID critical} 
follows by substituting \eqref{eqweight upper bound crude} with the upper bound from the lemma below.


\smallskip

\begin{lemma}\label{lemRadon Nikodym}
	Let $r \geq 6$. Then for every \. $x \. \in \.   B_{r-3} \setminus B_3$,
	\begin{equation}\label{eqRadon Nikodym}
	\big|\Eb_{\rho \sim \IUD} \. \Eb \big[\wt\big(x,\tau(\rho) \big)\big] \big|  \quad \leq \quad  \left(1-\frac{1}{58 \. 2^{58}}\right) \. |\ft(x)| \.. 
	\end{equation}
\end{lemma}

\smallskip

\begin{remark}
	The constant \. $(1-\frac{1}{58 \. 2^{58}})$ \. in \eqref{eqRadon Nikodym} is far from tight, but it is strictly less than 1, which is  sufficient for the  proof of Proposition~\ref{proprecurrence IID subcritical}.
%
\end{remark}

We build toward the proof of Lemma~\ref{lemRadon Nikodym} in the coming two subsections.

\medskip

\subsection{Admissible paths}
We fix an element \. $x \in B_{r-3} \setminus B_3$  and a rotor configuration $\rho$ throughout this subsection.  
	Let \. $p_{even}(\rho,x)$ \. be the probability of the $\Hor$--$\Ver$ walk $(Y_t,\zeta_t)_{t \geq 0}$, with initialization $((0,0),\rho)$, visiting $x$ a positive even number of times before terminating (at either the origin or $\partial B_r$),
\[ p_{even}(\rho,x) \quad := \quad \Pb\big[\.  |\{t \in \{1,\ldots, \tau\} \. \mid \. Y_t = x  \}| \. \in \. \{2,4,6,\ldots\} \. \big]\..   \]
The probability $p_{odd}(\rho,x)$ \. is defined analogously.

The main goal of this subsection is to prove the following lemma, which in turn will be used to prove Lemma~\ref{lemRadon Nikodym}\..

\smallskip

\begin{lemma}\label{lemRadon Nikodym 2}
	Let $r \geq 6$. Then for every $x \in B_{r-3} \setminus B_3$ and every rotor configuration $\rho$,
	\[ \frac{1}{58 \. 2^{58}}  \quad \leq \quad   \frac{p_{even}(\rho,x)}{p_{odd}(\rho,x)} \quad \leq \quad  58 \. 2^{58}\.. \]
\end{lemma}

\smallskip

We now build toward the proof of Lemma~\ref{lemRadon Nikodym 2}.

\smallskip

A \emph{word}  \. $w  = y_1 \ldots y_m$ \. is a finite string of vertices in $\Zb^2$.
We denote by \. $|w|  := m$ \. the length of the word $w$. 
For an arbitrary subset $S$ of $\Zb^2$, we write \. $w \in S$ \.  if  \. $y_1,\ldots, y_m$ \. are all contained in $S$.

Let  \. $(y,\eta)$ \. be an arbitrary vertex-and-rotor-configuration pair.
We define \. $\eta_t$ \. ($t \in \{0,1,\ldots,m\}$) \.
recursively by 
\begin{align*}
 (y_{0}, \eta_{0})  \quad & := \quad (y,\eta);\\
 \eta_{t}(x) \quad & := \quad
 \begin{cases}
 	\Hor & \text{ if } \ x= y_{t-1} \ \text{ and } \  y_t - y_{t-1} \ \in \ \{(1,0), -(1,0)\};\\
 	 \Ver & \text{ if } \ x= y_{t-1} \ \text{ and } \  y_t - y_{t-1} \ \in \ \{(0,1), -(0,1)\};\\
 	 \eta_{t-1}(x) & \text{ otherwise.} 
 \end{cases}
\end{align*}
We say that $w$ is an \emph{admissible path} for  \. $(y,\eta)$ \. if, for every $t \in \{1,\ldots,m\}$,
\begin{itemize}
	\item $y_{t}$ is a neighbor of $y_{t-1}$; and 
	\item \[\eta_{t-1}(y_{t-1}) \ = \ \begin{cases}
	\Ver & \text{ if } \  y_t - y_{t-1} \ \in \ \{(1,0), -(1,0)\};\\
	\Hor & \text{ if } \  y_t - y_{t-1} \ \in \ \{(0,1), -(0,1)\}.
	\end{cases} \]
\end{itemize}

\smallskip

The following image is useful:
The sequence $y_1, y_2\ldots, y_m$ is the locations of the walker after the first, second, \ldots,  $m$-th step
of an $\Hor$--$\Ver$ walk.
The sequence  \. $(y_0,\eta_{0})$\., \. $(y_0,\eta_{0})$\., \. \ldots\., \. $(y_{m},\eta_{m})$ \.
 is a feasible
 trajectory of the first $m$ steps of an $\Hor$--$\Ver$ walk.
 The probability for this trajectory to occur (for $q=1$) is equal to \. $\frac{1}{2^m}$ \. if  $w$ is an admissible path for $(y,\eta)$\.,
 and is equal to $0$ otherwise.
 
 \smallskip
 
 \begin{remark}
 	The admissible paths as defined above correspond to   \emph{legal executions}  for abelian networks introduced by Bond and Levine in \cite{BL},
 	with one difference being that the admissible path $y_1y_2\ldots y_m$ in this paper corresponds to the legal execution $y_0y_1\ldots y_{m-1}$ in the notation of Bond and Levine.
 	That is, our notation records the location of the walker \emph{after} $\Hor$--$\Ver$ moves are performed, while their notation records the location of the walker \emph{before} $\Hor$--$\Ver$ moves are performed.
 	In particular, the latter notation \emph{does not} record the final location of the walker,  and thus our notation is more suitable for the purpose of this paper.
 \end{remark}

\smallskip

We denote by $\Ac(x)$ the subset of $\Zb^2$ given by 
\[  \Ac(x) \ := \   \{ \. x +(i,j) \. \mid \.  i, j \in \{-2,-1,0,1,2\}\.,  \ (i,j) \neq (0,0)  \.  \}\.. \]
Note that $\Ac(x)$ is contained in $B_{r}$ since $x \in B_{r-3}$, \. $\Ac(x)$ does not contain $x$ by definition, and $\Ac(x)$ does not contain the origin since $x \notin B_3$\..

\smallskip

\begin{lemma}\label{lemword 1}
	Let $y$ be an element of $\Ac(x)$ and let $\eta$ be a rotor configuration.
	Then, for every  \. $y' \. \in \. \Ac(x)  \cap  \Ngh(y)$\., there exists a word $w$ such that
	\begin{itemize}
		\item $w$ is an admissible path for \. $(y,\eta)$ \. such that \. $y_{|w|}=y'$\.; and 
		\item $w \in \Ac(x)$ \. and \. $|w| \leq 5$\..
	\end{itemize} 
\end{lemma}

\smallskip

\begin{proof}
	We present only the proof for the case  \. $y=x+(2,2)$ \. and \. $y'=x+(1,2)$\.,
	as the proofs for other cases follow from a similar argument.
	There are four  possibilities to consider:
	 \begin{enumerate}
	 	\item $\eta(x+(2,2)) \ = \ \Ver$\..
	 	In this case, let \. $w=y_1$  with $y_1=x+(1,2)$\..
	 	Then $w$   has length 1 and satisfies the conditions in the lemma.
	 	\item  $\eta(x+(2,2)) \ = \ \Hor$ and $\eta(x+(2,1)) \ = \ \Hor$\..
	 	In this case, let $w=y_1y_2y_3$ with 
	 	\[y_1 \ = \ x+(2,1), \qquad  y_2 \ = \ x+(2,2),\qquad  y_3 \ = \ x+(1,2)\.. \] 	
	 	Then $w$   has length 3 and satisfies the conditions in the lemma.
	 	\item  $\eta(x+(2,2)) \ = \ \Hor$\., \  $\eta(x+(2,1)) \ = \ \Ver$\.,  \. and \.  $\eta(x+(1,1)) \ = \ \Hor$\..
	 	In this case, let $w=y_1y_2y_3$ with 
	 	\[y_1 \ = \ x+(2,1), \qquad  y_2 \ = \ x+(1,1),\qquad  y_3 \ = \ x+(1,2)\.. \]
	 	Then $w$   has length 3 and satisfies the conditions in the lemma. 
	 	\item  $\eta(x+(2,2)) \ = \ \Hor$\., \  $\eta(x+(2,1)) \ = \ \Ver$\.,  \. and \.  $\eta(x+(1,1)) \ = \ \Ver$\..
	 	In this case, let $w=y_1y_2y_3y_4y_5$ with 
	 	\[y_1 \ = \ x+(2,1), \qquad  y_2 \ = \ x+(1,1),\qquad  y_3 \ = \ x+(2,1), \qquad y_4 \ = \ x+(2,2),  \qquad y_5 \ = \ x+(1,2)\.. \]
	 	Then $w$   has length 5 and satisfies the conditions in the lemma. 		
	 \end{enumerate}
	(Note that, in all four possibilities, the word $w$ involves only four vertices, namely  \. $x+(1,1)$\., \. $x+(1,2)$\., \. $x+(2,1)$\., \.  and \. $x+(2,2)$\..)
	This completes the proof of the lemma.
\end{proof}

\smallskip

We denote by $\Dc(x)$ the outer layer of  $\Ac(x)$, 
\[   \Dc(x) \ := \ \{ \. x +(i,j)  \. \in \. \Ac(x) \. \mid \.  i \in \{-2,2\} \ \text{ or } \  j \in \{-2,2\}    \.  \}\..   \]

\smallskip

\begin{lemma}\label{lemword 2}
	Let $y$ be an element of $\Dc(x)$, and let $\eta$ be a rotor configuration.
	Then there exists a word $w$ such that 
	\begin{itemize}
		\item $w$ is an admissible path for \. $(y,\eta)$ \. such that \. $y_{|w|}=y$ \. and \. $\eta_{|w|}(y)=\eta(y)$\.; and 
		\item $w \in \Ac(x) \cup \{x\}$ \. and \. $|w| \leq 58$; and
		\item The vertex $x$ occurs exactly once in $w$.
	\end{itemize} 
\end{lemma}

\smallskip

\begin{proof}

	We define the word $w=w_1w_2\ldots w_{6}$ recursively as follows:
	\begin{enumerate}
		\item Let $w_1 \in \Ac(x)$ be  an admissible path for $(y,\eta)$ such that $y_{|w_1|}$ is a neighbor of $x$\..
		Such a word $w_1$ exists and \. $|w_1|\leq 15$ \.  by Lemma~\ref{lemword 1}\..
		\item Write \. $(y',\eta') \. := \.  (y_{|w_1|},\eta_{|w_1|})$\..
		Let $w_2$  be the empty word if the word $x$  is an admissible path for $(y',\eta')$.
		Otherwise, let $w_2 \in \Ac(x)$ be  an admissible path for
		$(y',\eta')$ such that $y'_{|w_2|} = y'$ and $y'$ occurs exactly once in $w_2$\..
		Such a word $w_2$ exists and \. $|w_2|\leq 1+5=6$ \.  by Lemma~\ref{lemword 1}\..
		\item Write \. $(y'',\eta'') \. := \.  (y'_{|w_2|},\eta'_{|w_2|})$\..
		It follows from the previous step that $w_3=x$ is an admissible path for $(y'',\eta'')$\., and that \. $y''_{|w_3|} =x$\..
		\item Write \. $(y^{(3)},\eta^{(3)}) \. := \.  (y''_{|w_3|},\eta''_{|w_3|})$\..
		Let $w_4$ be the word 
		\[ w_4 :=  \begin{cases}
			x+(1,0)  & \text{ if } \eta^{(3)}(x) \ = \  \Ver;\\
			x+(0,1)  & \text{ if } \eta^{(3)}(x) \ = \  \Hor.
			\end{cases}  \]
		 It follows from the definition that $w_4$ is an admissible path for  \. $(y^{(3)},\eta^{(3)})$\.,  and \. $y^{(3)}_{|w_4|}$ \. is neighbor of $x$, and hence is contained in $\Ac(x)$\..
		 \item Write \. $(y^{(4)},\eta^{(4)}) \. := \.  (y^{(3)}_{|w_4|},\eta^{(3)}_{|w_4|})$\..
		 Let $w_5 \in \Ac(x)$ be an admissible path for \. $(y^{(4)},\eta^{(4)})$ \.
		 such that \. $y^{(4)}_{|w_5|}= y$\..
		 Such a word $w_5$ exists and \. $|w_5|\leq 25$ \.  by Lemma~\ref{lemword 1}\..
		 		\item Write \. $(y^{(5)},\eta^{(5)}) \. := \.  (y^{(4)}_{|w_5|},\eta^{(4)}_{|w_5|})$\..
		 Let $w_6$  be the empty word if  for \. $\eta^{(5)}(y)=\eta(y)$\..
		 Otherwise, let $w_6 \in \Ac(x)$ be  an admissible path for
		 $(y^{(5)},\eta^{(5)})$  such that $y^{(5)}_{|w_6|} = y$ and $y$ occurs exactly once in $w_6$\..
		 Such a word $w_6$ exists and \. $|w_6|\leq 10$ \.  by Lemma~\ref{lemword 1}\..
		 It follows from the definition that \.  $y^{(5)}_{|w_6|}=y$ \. and \.   $\eta^{(5)}_{|w_6|}(y)=\eta(y)$\..
	\end{enumerate}
	See Figure~\ref{figword} for an illustration of the construction of $w$.
	
\begin{figure}[hbt]
	\centering
	\begin{tabular}{c@{\hskip 0.3in}c @{\hskip 0.3in} c  @{\hskip 0.3in} c}
		\includegraphics[width=0.2\linewidth]{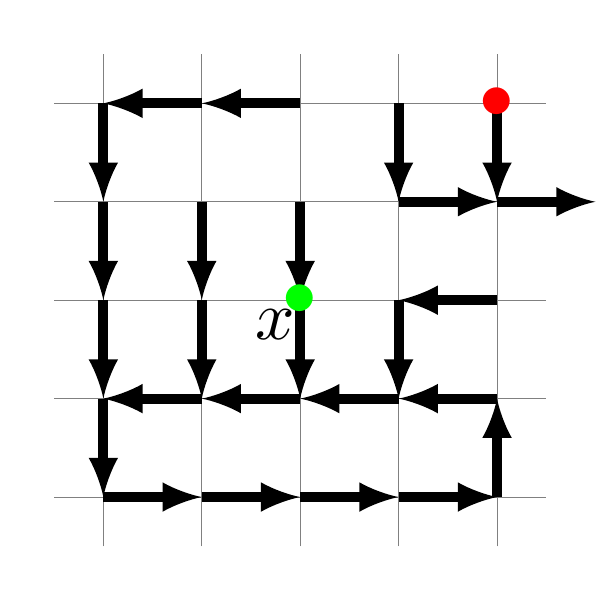}  & 
		\includegraphics[width=0.2\linewidth]{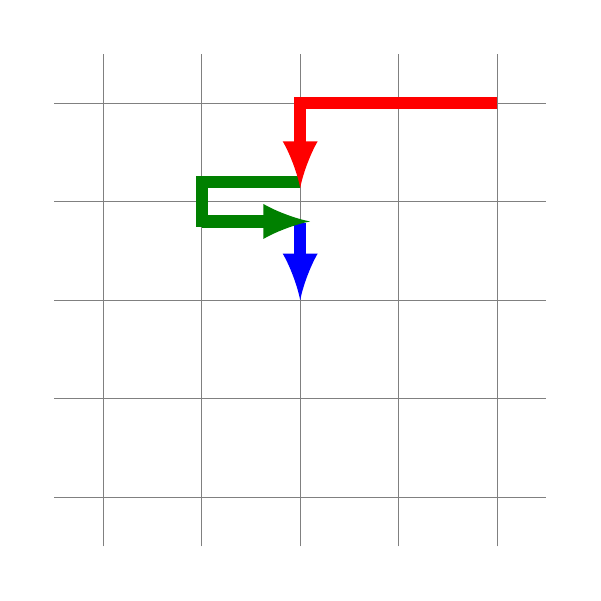}  & 
		\includegraphics[width=0.2\linewidth]{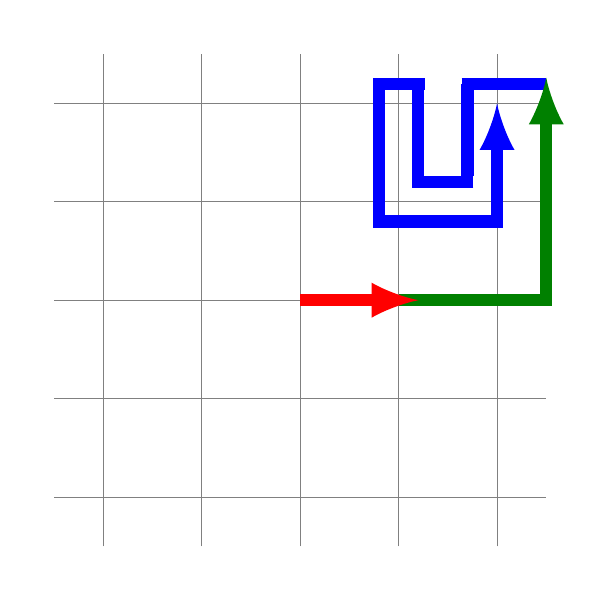}  & 
		\includegraphics[width=0.2\linewidth]{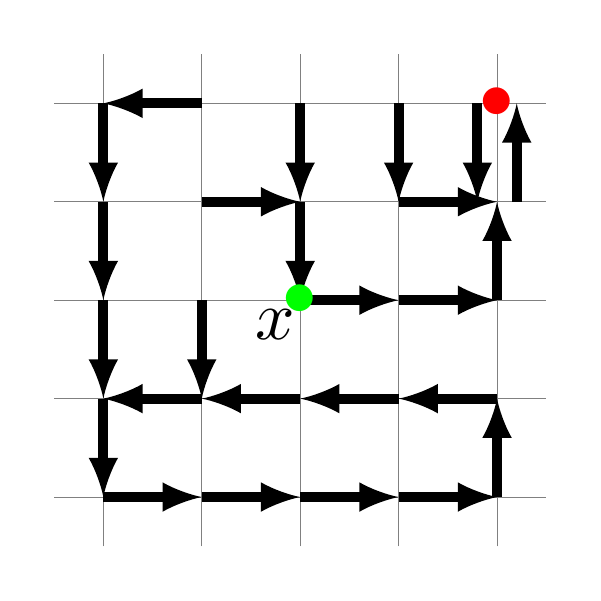}  \\
		(a) & (b) & (c) & (d)
	\end{tabular}
	\caption{(a) The pair $(y,\eta)$, with $y$ given by the red bullet and with $\eta$ restricted to $\Ac(x) \cup \{x\}$.
		(b) The trajectory of the word $w_1$ (in red), $w_2$ (in green), and $w_3$ (in blue).
		(c) The trajectory of the word $w_4$ (in red), $w_5$ (in green), and $w_6$ (in blue).
		(d) The final pair $(y_{|w|},\eta_{|w|})$. Note that $\eta_{|w|}(y) = \eta(y)$ and 
		$\eta_{|w|}(x) \neq \eta(x)$ as required by Lemma~\ref{lemword 2}.
			}
	\label{figword}
\end{figure}

	Now note that $w=w_1\ldots w_6$ is an admissible path for $(y,\eta)$ since it is a concatenation of admissible paths.
	Also note that 
	\[  y_{|w|} \ = \ y^{(5)}_{|w_6|} \ = \ y\.; \qquad \text{ and } \qquad  \eta_{|w|}(y) \ = \ \eta^{(5)}_{|w_6|}(y)  \ = \ \eta(y)\.. \]
	Furthermore note that $w \in \Ac(x) \cup \{x\}$ by construction, and $x$ occurs exactly once in $w$ (namely as $w_3$).
	Finally note that the length of $w$ satisfies 
	\[ |w| \quad = \quad \sum_{i=1}^{6} |w_1| \quad \leq \quad 15 \.+\. 6 \. + \. 1 \. + \. 1 \.+ \. 25 \. + \. 10 \quad = \quad 58\..   \] 
	This proves the lemma. 
\end{proof}

\smallskip

Let $\rho$ be an arbitrary rotor configuration.
We denote by \. $J(\rho):= J(\rho,k,r)$ \. the set of admissible paths $w=x_1\ldots x_m$ for $((0,0),\rho)$, such that, exactly one of the following scenarios occur:
\begin{equation}\label{eqTau}\tag{Tau}
\begin{split}
	&\text{$x_m=(0,0)$, and $(0,0)$ occurs in $w$ exactly $k$ times; or}\\
	&\text{$x_m \in \partial B_r$, and all other vertices in $w$ are not contained in $\partial B_r$\..}
\end{split}
\end{equation}
Described in words, \.  $J(\rho)$ \. consists of words that record the feasible locations of the walker until the hitting time $\tau$,  for the $\Hor$--$\Ver$ walk with initialization $((0,0),\rho)$. 

We denote by  \. $J_0(\rho,x):= J_0(\rho,x,k,r)$ \. the set of words $w$ in $J(\rho)$ such that $x$ never occurs in $w$,
and by   \. $J_1(\rho,x):= J_1(\rho,x,k,r)$ \. the set of words $w$ in $J(\rho)$ such that $x$ occurs at least once in $w$.
We now define the map \. $\phi:J_1(\rho,x) \to J_1(\rho,x)$ \. as follows:
\begin{itemize}
	\item Let \. $w = y_1\ldots y_m \in J_1(\rho,x)$\..
	Since $x$ occurs at least once in $w$, it follows  that some vertex in $\Dc(x)$ also occurs at least once in $w$.
	Let $\ell:=\ell(w)$ be the largest integer such that $y_{\ell} \in \Dc(x)$\..
	Note that \. $\{y_{\ell+1}, \ldots, y_m\} \ts \cap \ts (\Ac(x) \cup \{x\}) =\varnothing$ \. and \. $x \in B_{r-3} \setminus B_3$.
	\item 
	We define \. $w_1 \. := \. y_1y_2\ldots y_{\ell}$ \. and \. $w_3 \. := \. y_{\ell+1}y_{\ell+2}\ldots y_{m}$\.. 
		\item Let \. $(y',\eta'):= (y_\ell,\eta_\ell)$ \. (with respect to the word $w_1$ and  initialization \. $(y_0,\eta_0)=((0,0),\rho)$)\..
		Note that $y' \in \Dc(x)$ by definition.
		We define $w_2$ to be an admissible path for $(y',\eta')$ such that \. $y'_{|w_2|}=y'$ \. and \. $\eta'_{|w_2|}(y')=\eta'(y')$\..
		We also require that \. $w \in \Ac(x) \cup \{x\}$ \. and that $x$ occurs exactly once in $w$.
		Such a word $w_2$ exists and $|w_2| \leq 58$ by Lemma~\ref{lemword 2}\..
	\item We define $\phi(w) := w_1 w_2 w_3$\..
\end{itemize}
\smallskip

\begin{lemma}\label{lemphi well defined}
	For every \. $w \in J_1(\rho,x)$\., we have that $\phi(w)$ is contained in $J_1(\rho,x)$.
	Furthermore, x occurs in $\phi(w)$ exactly one more time than in $w$, and \. $|\phi(w)| \. \leq \. |w| + 58$\..
\end{lemma}

\smallskip

\begin{proof}
	We first show that $\phi(w)$ is an admissible path for $((0,0),\rho)$\..
	First we have  $w_1$ is an admissible path for $((0,0),\rho)$ by definition,
	and $w_2$  is an admissible path for   \. $(y',\eta')$\..
	We write \. $(y'',\eta''):= (y'_{|w_2|}, \eta'_{|w_2|})$\..
	Now note that \. $y''$ is equal to \. $y'$\.,  and \. $\eta'$ agrees with $\eta''$ on every vertex outside of \. $\Ac(x) \cup \{x \}$\..
	Also note that $w_3$ is an admissible path for $(y',\eta')$, and none of the vertices in \. $\Ac(x) \cup \{x \}$ \. occurs in $w_3$ (by the maximality of $\ell$).
	It then follows from these two observations that $w_3$ is  an admissible path for \. $(y'',\eta'')$\..
	Hence \. $\phi(w) = w_1 w_2 w_3$ \. is a concatenation of  admissible paths,
	and we conclude $\phi(w)$ is an admissible path for $((0,0),\rho)$\..
	
	We now show that $\phi(w)$ satisfies \eqref{eqTau}.
	Indeed, note that \. $\Ac(x) \cup\{x\}$ \. contains neither the origin nor vertices from $\partial B_r$\., since $x \in B_{r-3} \setminus B_3$.
	Also note that \. $w=w_1w_3$ \. satisfies \eqref{eqTau} by assumption.
	It then follows from these two observations that \. $\phi(w)=w_1w_2w_3$ \. satisfies \eqref{eqTau}.
	
	Now  
	note that, $x$ occurs in $\phi(w)=w_1w_2w_3$ exactly one more time than in $w=w_1w_3$, since, $x$ occurs exactly once in $w_2$.
	Also note that
	\[  |\phi(w) |  \quad = \quad  |w_1|+ |w_2| + |w_3| \quad \leq \quad  |w_1|+ 58 + |w_3|  \quad = \quad |w|+58\..\]
	This completes the proof of the lemma. 
\end{proof}

\smallskip

\begin{lemma}\label{lemphi preimage}
	Let \. $w \in \phi\big(J_1(x,\rho)\big)$\..
	Then the preimage of $w$ under $\phi$ contains at most 58 elements.
\end{lemma}

\smallskip

\begin{proof}
	Let $w= x_1\ldots x_m $, and let $\ell$ be the largest integer such that \. $x_\ell \in \Dc(x)$\..
	Let $w'$ be an arbitrary element of the preimage of $w$ under $\phi$.
	It then follows from the construction of $\phi$ that 
	\[w' = x_1 \ldots x_{\ell-i} \  x_{\ell+1}\ldots x_m \qquad \text{ for some } i \in \{1,\ldots, 58\}\.. \]
	It then follows that the preimage of $w$ under $\phi$ contains at most 58 elements, as desired.
\end{proof}

\smallskip

\smallskip

We are now ready to present the proof of Lemma~\ref{lemRadon Nikodym 2}.

\smallskip

\begin{proof}[Proof of Lemma~\ref{lemRadon Nikodym 2} ]
	Note that we have
	\begin{align}\label{eqecho 1}
	p_{even}(\rho,x) \quad & = \quad  \sum_{\substack{w \in J_1(\rho,x);\\ K_x(w) \. \equiv \. 0 \text{ mod } 2}}  2^{-|w|};
	\qquad 
	p_{odd}(\rho,x) \quad  = \quad  \sum_{\substack{w \in J_1(\rho,x);\\ K_x(w) \. \equiv \. 1 \text{ mod } 2}}  2^{-|w|}\.,
	\end{align}
	where $K_x(w)$ is the number of visits to $x$ by the walk with admissible path $w$.
	Now  note that,  by Lemma~\ref{lemphi well defined},
	\begin{equation}\label{eqecho 2}
		p_{even}(\rho,x) \quad \leq \quad  2^{58}\sum_{\substack{w \in J_1(\rho,x);\\ K_x(w) \. \equiv \. 0 \text{ mod } 2}}  2^{-|\phi(w)|}\..
	\end{equation}

	On the other hand, again by Lemma~\ref{lemphi well defined},
	\begin{equation}\label{eqecho 3}
		\{\. \phi(w) \. \mid \.   w \in J_1(\rho,x)\.; K_x(w) \. \equiv \. 0 \text{ mod } 2 \. \}	 \ \subseteq \
		\{\.  w  \in J_1(\rho,x) \. \mid \.    K_x(w) \. \equiv \. 1 \text{ mod } 2 \.\}\.,
	\end{equation}
	and furthermore each element in the right side of \eqref{eqecho 3} appears at most 58 times in the left side of \eqref{eqecho 3} by Lemma~\ref{lemphi preimage}. 
	Plugging \eqref{eqecho 3} into \eqref{eqecho 2}, 
	we get 
	\[ 	p_{even}(\rho,x)  \quad \leq \quad  58\. 2^{58} \. \sum_{\substack{w \in J_1(\rho,x);\\ K_x(w) \. \equiv \. 1 \text{ mod } 2}}  2^{-|w|} \quad = \quad 58\. 2^{58} \. p_{odd}(\rho,x)\.. \]
	By an analogous argument we also have \. $p_{odd}(\rho,x) \ \leq \ 58\. 2^{58} \. p_{even}(\rho,x)$\., and the lemma now follows.
\end{proof}

\medskip

\subsection{Proof of Lemma~\ref{lemRadon Nikodym}}

We will first prove the following lemma.

\smallskip

\begin{lemma}\label{lemRadon Nikodym 3}
	Let $r \geq 6$. Then for every \. $x \. \in \.   B_{r-3} \setminus B_3$,
	\begin{equation*}
	\big|\Eb_{\rho \sim \IUD} \. \Pb\big[\zeta_{\tau}(x) = \Hor \big] \ - \ 	\Eb_{\rho \sim \IUD} \. \Pb\big[\zeta_{\tau}(x) = \Ver \big]\big| \quad \leq \quad \left(1-\frac{1}{58 \. 2^{58}}\right)\..
	\end{equation*} 
\end{lemma}

\begin{proof}
	For every rotor configuration $\rho$, we write 
	\[p_0(\rho,x) \quad := \quad \sum_{w \in J_0(\rho,x)} 2^{-|w|}\.,  \]
	the probability that the  $\Hor$--$\Ver$ walk  with initial rotor configuration $\rho$ never visits $x$  before terminating (at either the origin or $\partial B_r$).
	In particular, $p_{0}(\rho,x)$ does not depend on the rotor $\rho(x)$ at $x$ since the walker never visits $x$.
	Now note that,
	\begin{equation*}
	\begin{split}
		\Pb\big[\zeta_{\tau}(x)  = \Hor \big] \quad = \quad  & \satu\{\rho(x) \. = \. \Hc  \}  \, p_o(\rho,x) \ + \ \satu\{\rho(x) \. = \. \Hor  \} \, p_{even}(\rho,x) \\& + \quad \satu\{\rho(x) \. = \. \Ver  \} \, p_{odd}(\rho,x)\.. 
	\end{split}
	\end{equation*}
	On the other hand, since $p_{o}(\rho,x)$ does not depend on $\rho(x)$,
	\begin{align*}
		 \Eb_{\rho \sim \IUD} \big[ \satu\{\rho(x) \. = \. \Hor  \}  \, p_o(\rho,x) \big]  \quad & = \quad    \Eb_{\rho \sim \IUD} \big[ \satu\{\rho(x) \. = \. \Hor  \}\big] \, \Eb_{\rho \sim \IUD} \big[ p_o(\rho,x) \big] \\
		 & = \quad \frac{1}{2} \. p_o(\rho,x)\..
	\end{align*}

	Combining the two observations above, we get
	\begin{equation}\label{eqfoxtrot 1}
	\begin{split}
		\Eb_{\rho \sim \IUD} \. \Pb\big[\zeta_{\tau}(x) = \Hor \big] \quad = \quad & \frac{1}{2} \. p_o(\rho,x) 
		\ + \ \Eb_{\rho \sim \IUD} \big[ \satu\{\rho(x) \. = \. \Hor  \} \, p_{even}(\rho,x) \big]\\
		&  + \ \Eb_{\rho \sim \IUD} \big[ \satu\{\rho(x) \. = \. \Ver  \} \, p_{odd}(\rho,x) \big]\..
	\end{split}
	\end{equation}
	By an analogous argument, we have
		\begin{equation}\label{eqfoxtrot 2}
	\begin{split}
	\Eb_{\rho \sim \IUD} \. \Pb\big[\zeta_{\tau}(x) = \Ver \big] \quad = \quad & \frac{1}{2} \. p_o(\rho,x) 
	\ + \ \Eb_{\rho \sim \IUD} \big[ \satu\{\rho(x) \. = \. \Hor  \} \, p_{odd}(\rho,x) \big]\\
	&  + \ \Eb_{\rho \sim \IUD} \big[ \satu\{\rho(x) \. = \. \Ver  \} \, p_{even}(\rho,x) \big]\..
	\end{split}
	\end{equation}
	
	We now apply Lemma~\ref{lemRadon Nikodym 2} to \eqref{eqfoxtrot 1}, and we get
	\begin{equation}\label{eqfoxtrot 3}
	\begin{split}
		\Eb_{\rho \sim \IUD} \. \Pb\big[\zeta_{\tau}(x) = \Hor \big] \quad \geq  \quad & \frac{1}{2} \. p_o(\rho,x) 
		\ + \ \frac{1}{58 \. 2^{58}}\. \Eb_{\rho \sim \IUD} \big[ \satu\{\rho(x) \. = \. \Hor  \} \, p_{odd}(\rho,x) \big]\\
		&  + \ \frac{1}{58 \. 2^{58}}\.  \Eb_{\rho \sim \IUD} \big[ \satu\{\rho(x) \. = \. \Ver  \} \, p_{even}(\rho,x) \big]\..
	\end{split}
	\end{equation}
	Taking the difference between \eqref{eqfoxtrot 2} and \eqref{eqfoxtrot 3}, 
	we get
	\begin{equation}\label{eqfoxtrot 4}
	\begin{split}
			& \Eb_{\rho \sim \IUD} \. \Pb\big[\zeta_{\tau}(x) = \Ver \big] \ - \ 	\Eb_{\rho \sim \IUD} \. \Pb\big[\zeta_{\tau}(x) = \Hor \big] \\ 
			& \leq \left(1-\frac{1}{58 \. 2^{58}}\right)  \. \bigg(  \Eb_{\rho \sim \IUD} \big[ \satu\{\rho(x) \. = \. \Hor  \} \, p_{odd}(\rho,x) \big] \\
			&  \phantom{ \leq \left(1-\frac{1}{58 \. 2^{58}}\right)  \. \bigg(    } \ + \ \Eb_{\rho \sim \IUD} \big[ \satu\{\rho(x) \. = \. \Ver  \} \, p_{even}(\rho,x) \big]   \bigg)\..
	\end{split}
	\end{equation}
	Noting that the last term in \eqref{eqfoxtrot 4} is less than the right side of \eqref{eqfoxtrot 2}, we get 
	\begin{equation*}
	\begin{split}
		& \Eb_{\rho \sim \IUD} \. \Pb\big[\zeta_{\tau}(x) = \Ver \big] \ - \ 	\Eb_{\rho \sim \IUD} \. \Pb\big[\zeta_{\tau}(x) = \Hor \big] \quad  \leq  \quad
		\left(1-\frac{1}{58 \. 2^{58}}\right)  \. 	\Eb_{\rho \sim \IUD} \. \Pb\big[\zeta_{\tau}(x) = \Ver \big]\..
	\end{split}
	\end{equation*}
	Noting that \. $\Pb\big[\zeta_{\tau}(x) = \Ver \big]$ \. is at most equal to 1, we
	then have 
	\begin{equation}\label{eqfoxtrot 5}
	 \Eb_{\rho \sim \IUD} \. \Pb\big[\zeta_{\tau}(x) = \Ver \big] \ - \ 	\Eb_{\rho \sim \IUD} \. \Pb\big[\zeta_{\tau}(x) = \Hor \big] \quad \leq \quad \left(1-\frac{1}{58 \. 2^{58}}\right)\..
	\end{equation} 
	By an analogous argument, we also have
		\begin{equation}\label{eqfoxtrot 6}
	\Eb_{\rho \sim \IUD} \. \Pb\big[\zeta_{\tau}(x) = \Hor \big] \ - \ 	\Eb_{\rho \sim \IUD} \. \Pb\big[\zeta_{\tau}(x) = \Ver \big] \quad \leq \quad \left(1-\frac{1}{58 \. 2^{58}}\right)\..
	\end{equation} 
	The lemma now follows by combining~\eqref{eqfoxtrot 5} and \eqref{eqfoxtrot 6}.
	\end{proof}
	
	\smallskip
	
	\begin{proof}[Proof of Lemma~\ref{lemRadon Nikodym}]
	We have by the definition of the weight function that
	\begin{align*}
	\Eb_{\rho \sim \IUD} \. \Eb \big[\wt\big(x,\tau(\rho) \big)\big]
	\quad = \quad 	\Eb_{\rho \sim \IUD} \. \Pb\big[\zeta_{\tau}(x) = \Hor \big]\, 
	\ft_{\Hor}(x)  \ + \ \Eb_{\rho \sim \IUD} \. \Pb\big[\zeta_{\tau}(x) = \Ver \big]\, 
	\ft_{\Ver}(x)\..  
	\end{align*}
	It then follows from \eqref{equation: weight harmonic} and \eqref{eqweight max definition} that 
	\begin{align*}
	 \big|	\Eb_{\rho \sim \IUD} \. \Eb \big[\wt\big(x,\tau(\rho) \big)\big]  \big|
	\quad   = \quad & \big|\Eb_{\rho \sim \IUD} \. \Pb\big[\zeta_{\tau}(x) = \Hor \big] \ - \ 	\Eb_{\rho \sim \IUD} \. \Pb\big[\zeta_{\tau}(x) = \Ver \big]\big| |\ft(x)|\..
	\end{align*}
	An application of Lemma~\ref{lemRadon Nikodym 3} to the last equation yields \eqref{eqRadon Nikodym}, as desired.
\end{proof}

\smallskip

%

\begin{proof}[Proof of Proposition~\ref{proprecurrence IID critical}]	
	We apply the same calculations as in the proof of Proposition~\ref{proprecurrence IID subcritical} but with 
	\eqref{eqRadon Nikodym} substituting
	\eqref{eqweight upper bound crude} (with details omitted for brevity), and  we get
	\[ \Eb_{\rho \sim \IUD} \big[p_{\rec}(\rho)  \big]  \quad \geq  \quad \frac{1}{58 \. 2^{58}}  \quad > \quad 0\..
	\]
	It now follows from Proposition~\ref{proposition: zero-one law IUD} that 
	\. $p_{\rec}(\rho) \. = \. 1$  \. for almost every rotor configuration $\rho$ sampled from $\IUD$, as desired.
\end{proof}

\smallskip

\begin{proof}[Proof of Theorem~\ref{thmrecurrence IID}]
	The theorem follows from combining Proposition~\ref{proprecurrence IID subcritical} (for the case $q<1$) and Proposition~\ref{proprecurrence IID critical} (for the case $q=1$).
\end{proof}

\bigskip

\section{Concluding remarks}\label{secconcluding remarks}


\subsection{Recurrence of other  rotor configurations}\label{subsec existence of recurrent configurations}
The techniques used in this paper is quite robust to changes of the initial rotor configuration, and in some cases one can even get a stronger result for specific rotor configurations, e.g.,
\begin{itemize}
	\item For $q\in [\frac{1}{2},1]$, the following rotor configuration is recurrent:
	\[ \rho_{\text{Box}}(x)  \quad := \quad 
		\begin{cases}
		x+(0,1) & \text{ if }   \arg(x) \in \big(-\frac{\pi}{4}, \frac{\pi}{4}\big];\\
		x-(1,0) & \text{ if }   \arg(x) \in \big(\frac{\pi}{4}, \frac{3\pi}{4}\big];\\
		x-(0,1) & \text{ if }   \arg(x) \in \big(\frac{3\pi}{4}, \pi \big] \ \text{ or } \ \arg(x) \in \big( -\pi,  -\frac{3\pi}{4}\big];\\
		x+(1,0) & \text{ if }   \arg(x) \in \big(-\frac{3\pi}{4}, -\frac{\pi}{4}\big].
		\end{cases}  \]
	\item For $q\in (0,\frac{1}{2}]$, the following rotor configuration is recurrent:
	\[ \rho_{\text{Line}}(x)  \quad := \quad 
		\begin{cases}
		x+(1,0) & \text{ if }   \arg(x) \in \big(-\frac{\pi}{4}, \frac{\pi}{4}\big];\\
		x+(0,1) & \text{ if }   \arg(x) \in \big(\frac{\pi}{4}, \frac{3\pi}{4}\big];\\
		x-(1,0) & \text{ if }   \arg(x) \in \big(\frac{3\pi}{4}, \pi \big] \ \text{ or } \ \arg(x) \in \big( -\pi,  -\frac{3\pi}{4}\big];\\
		x-(0,1) & \text{ if }   \arg(x) \in \big(-\frac{3\pi}{4}, -\frac{\pi}{4}\big].
		\end{cases}  \]
\end{itemize}
See Figure~\ref{Figrotor configurations} for an illustratrion of \. $\rho_{\text{Box}}$,  \.  $\rho_{\text{Line}}$, and \. $\rho_{\text{Alternating}}$\..
The rotor configuration  $\rho_{\text{Box}}$ is notably also 
a recurrent configuration~\cite{AH2} and a configuration with minimal range~\cite{FLP}  for the clockwise rotor walk.

\begin{figure}[hbt]
	\centering
	\begin{tabular}{c@{\hskip 0.6in}c @{\hskip 0.6in} c @{\hskip 0.6in} c}
		\includegraphics[width=0.2\linewidth]{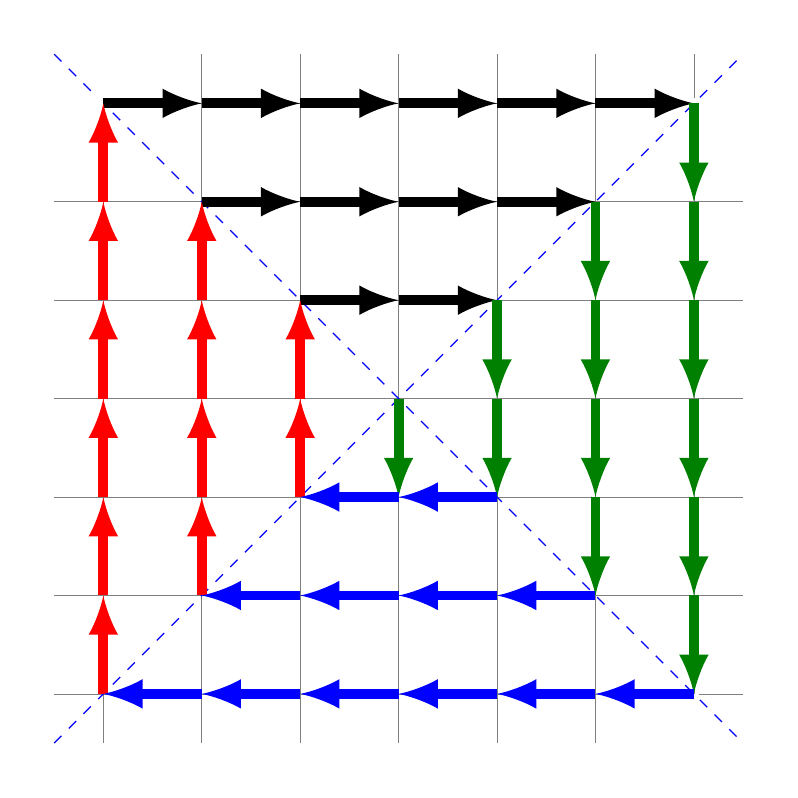}  & 
		 \includegraphics[width=0.2\linewidth]{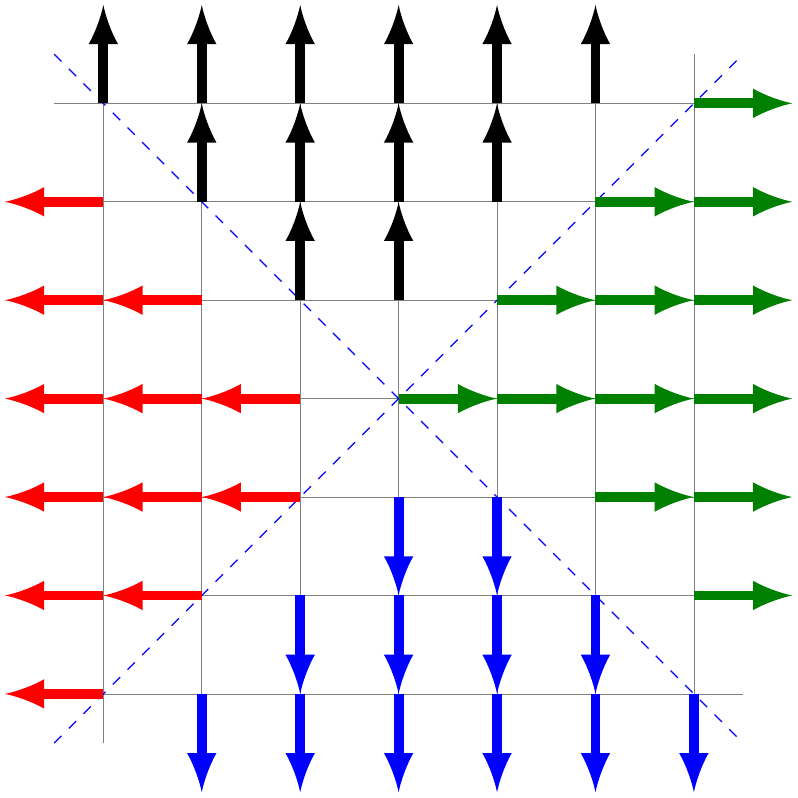}   &
		 \includegraphics[width=0.2\linewidth]{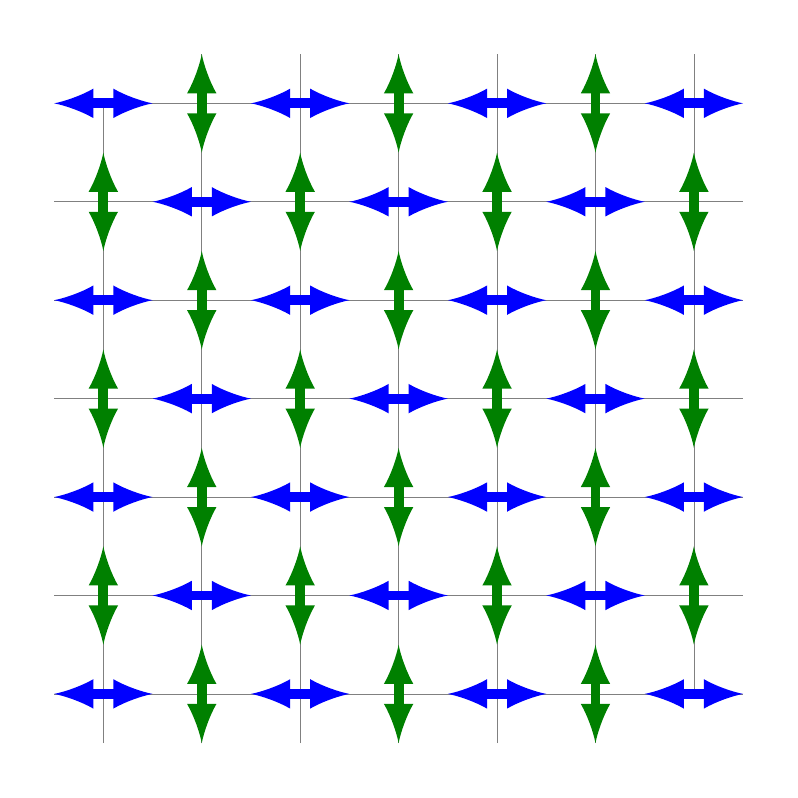} &
		\includegraphics[width=0.2\linewidth]{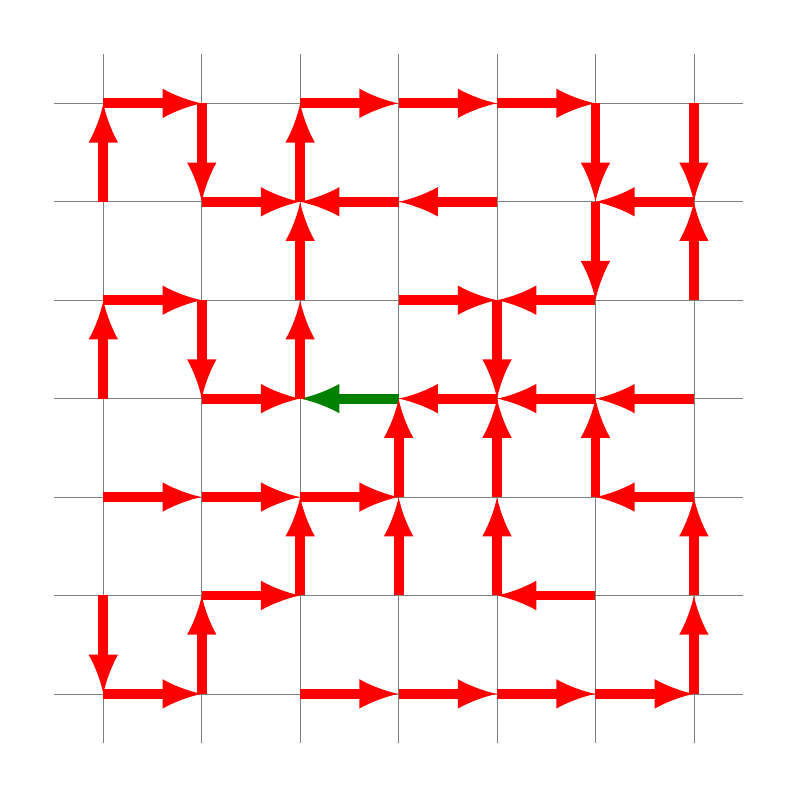}  \\
		(a) & (b) & (c) & (d)
	\end{tabular}
	\caption{(a) The  configuration $\rho_{\text{Box}}$.
			 (b) The  configuration $\rho_{\text{Line}}$.
			 (c) The configuration $\rho_{\text{Alternating}}$.
		 	(d) The  configuration $\rho_{\text{USTP}}$.
	 		Note that, for $\Hor$--$\Ver$ walks,
	 		the up-arrow ($\uparrow$) and the down-arrow ($\downarrow$) is equivalent to the label $\Ver$~($\updownarrow$), while  the left-arrow ($\leftarrow$) and the right-arrow ($\rightarrow$) is equivalent to the label $\Hor$ ($\leftrightarrow$).}
	\label{Figrotor configurations}
\end{figure}

%

On the other hand, we believe that our results in Theorem~\ref{thmrecurrence many walker} and Theorem~\ref{thmrecurrence IID}  are likely not tight.
Indeed, the techniques used in the proof of Theorem~\ref{thmrecurrence IID} could be used  to derive the recurrence for $\IID$ uniform rotor configuration on the four directions for a  slightly larger regime
$q \in \big(\frac{1}{3}-\varepsilon, 1\big]$, with \. $\varepsilon \. \approx \. 1.5 \times 10^{-19}$\..
In fact, we believe that the following stronger claim is true.

\smallskip

\begin{conjecture}\label{conjhv recurrence}
		Let $q \in (0,1]$.
	Then, for every initial rotor configuration,
	the corresponding $\Hor$--$\Ver$ walk with a single walker is recurrent a.s..
\end{conjecture}

\smallskip

Note that the recurrence regime $q \in (0,1]$ in Conjecture~\ref{conjhv recurrence} is the best one could hope for,
as  there exist  transient rotor configurations when  $q=0$ (see  Section~\ref{subsecq0} below).



The real obstacle in proving Conjecture~\ref{conjhv recurrence} is the lack of understanding of the law of  the rotor configuration $\rho_t$ at the $t$-th step of the walk (we speculate on the law of $\rho_t$  in Problem~\ref{queststationarity} below).
Indeed, with the exception of the proof of Proposition~\ref{proprecurrence IID critical},
we always use the following rudimentary upper bound: for all $t \geq 0$ and $x \in \Zb^2$, 
\[ \big | \. \Pb [\rho_t(x) \ = \ \Hor ]  \ - \  \Pb [\rho_t(x) \ = \ \Ver ] \. \big |  \quad \leq \quad 1\.. \]
Thus developing a better upper bound for the inequality above would (e.g., an $o(1)$ upper bound)  constitute a natural first step  in solving  Conjecture~\ref{conjhv recurrence}. 

\smallskip

\subsection{The case $q=0$}\label{subsecq0}
None of our main results apply to $\Hor$--$\Ver$ walks with $q=0$, as the martingale in Definition~\ref{defmartingale} is not well defined.
In fact, there are  rotor configurations for this walk that are transient, regardless of the number of walkers~(c.f., Theorem~\ref{thmrecurrence many walker}).
Indeed,
it is straightforward to check that, for $\rho_{\text{Box}}$, each walker visits every vertex only finitely many times.
On the other hand, it is straightforward to show that the following rotor configuration is recurrent for this walk:
\[\rho_{\text{Alternating}}(a,b) \ := \ 
\begin{cases}
	\Hor & \text{ if } b-a \ \equiv \ 0 \text{ mod } 2;\\
	\Ver & \text{ if } b-a \ \equiv \ 1 \text{ mod } 2,
\end{cases}  
\]
%
%
as  the even steps of this walk $(X_{2t})_{t \geq 0}$ is a simple random walk on $\Zb^2$ with  each step being sampled uniformly from $\{(\pm 1, \pm 1)\}$\..  
(See Figure~\ref{Figrotor configurations} for an illustration of $\rho_{\text{Alternating}}$.)


Suppose now that    $\rho$ is sampled from the $\IID$ uniform measure on $\{\Hor,\Ver\}$.
Then the $\Hor$--$\Ver$ walk is not recurrent, as there are infinitely many vertices of $\Zb^2$ that are visited only finitely many times a.s..
Indeed, this is because $x\in \Zb^2$ will never be visited if the following property is satisfied:
\[ \rho(x+(1,0)) \ = \  \rho(x-(1,0))  \ = \  \Ver; \qquad \rho(x+(0,1)) \ = \  \rho(x-(0,1)) \ = \ \Hor\.,\]
and, by the Borel-Cantelli lemma, there are infinitely many vertices in $\Zb^2$ with this property.
On the other hand, 
there always exist some $x \in \Zb^2$ for which $x$ is visited infinitely many times by this walk; see the proof in the discussion after Lemma~2.5 in \cite{HS14}.
Note that
 the dichotomy in Lemma~\ref{lemma: dichotomy recurrence transience} does not apply here as the corresponding stack is not regular.

%
%

%

\medskip

\subsection{Stationary distribution and scaling limit}
 Consider the rotor configuration \. $\rho_{\text{USTP}}$, where  
the rotors at \. $\Zb^2 \setminus \{(0,0)\}$ \. form a (random) uniform spanning tree directed toward the origin~(see \cite{Pem,BLPS}),  and the rotor at the origin is sampled uniformly from the 
neighbors of the origin, independently from the uniform spanning tree (See Figure~\ref{Figrotor configurations} for an illustration of $\rho_{\text{USTP}}$).
It was shown in \cite[Theorem~1.1]{CGLL} that $\rho_{\text{USTP}}$  is  \emph{stationary} with respect to the   \emph{scenery process} of the single-walker $\Hor$--$\Ver$ walk.
That is, if the initial rotor configuration $\rho_0$ is equal to $\rho_{\text{USTP}}$,
then the rotor configuration \. $\rho_t(\cdot -X_t)$ \. at the $t$-th step of walk, observed from the viewpoint of the walker, is equal in distribution to  
$\rho_{\text{USTP}}$\..
It remains to be seen if the following stronger claim of stationarity is true.

\smallskip

\begin{problem}\label{queststationarity}
Let $q>0$. and let \. $(X_t,\rho_t)_{t \geq 0}$ \. be an $\Hor$--$\Ver$ walk with a single walker with an arbitrary initial rotor configuration.
Show that that  $\rho_t(\cdot - X_t)$ converges weakly to $\rho_{\text{USTP}}(\cdot)$ as $t \to \infty$.
That is, for every $x_1,\ldots, x_m \in V$ and every \. $d_1,\ldots, d_m \in \{\Hor,\Ver\}$, show that 
\[ \Pb\big[\. \rho_t(x_1-X_t) = d_1 \. , \. \ldots \. , \. \rho_t(x_m-X_t) = d_m \. \big] \quad \overset{t \to \infty}{\longrightarrow} \quad \Pb\big[\. \rho_{\text{USTP}}(x_1) = d_1 \. , \. \ldots \. , \. \rho_{\text{USTP}}(x_m) = d_m \. \big]\..    \]
\end{problem}

\smallskip


The fact that $\rho_{\text{USTP}}$ is stationary was used in \cite[Theorem~1.2]{CGLL} to show that the quenched scaling limit of the $\Hor$--$\Ver$ walk with the initial rotor configuration $\rho_{\text{USTP}}$ is the standard Brownian motion in $\Rb^2$.
Simulations suggest that we will obtain the same scaling limit when the  initial rotor configuration is sampled from the $\IID$ uniform measure  on $\{\Hor,\Ver\}$ (see Figure~\ref{figure: simple random walk vs rotor walk}).

\smallskip

\begin{conjecture}\label{conjquenched}
	Let $q>0$, and let \. $(X_t,\rho_t)_{t \geq 0}$ \. be an $\Hor$--$\Ver$ walk with a single walker with the  initial rotor configuration $\rho$ sampled from the $\IID$ uniform measure  on $\{\Hor,\Ver\}$.
	Then the quenched scaling limit for this walk is the standard Brownian motion $(B_t)_{t \geq 0}$ in $\Rb^2$\..
	That is to say, for almost every $\rho$ sampled from the $\IID$ uniform measure  on $\{\Hor,\Ver\}$,
	\[ \frac{1}{\sqrt{n}} (X_{\lfloor nt \rfloor})_{t \geq 0}    \quad \overset{n \to \infty}{\Longrightarrow}  \quad  (B_t)_{t \geq 0},\]
	with the convergence being the  weak convergence in the
	Skorohod space $D_{\Rb^2}[0,\infty)$.
\end{conjecture}

Note that the case $q=0$ in Conjecture~\ref{conjquenched}  is a special case of the result of  Berger and Deuschel~ \cite[Theorem~1.1]{BD}, who proved a quenched invariance principle for a large family of random walks in random environments (note that the $\Hor-\Ver$ walk is  not a random walk in random environment anymore when $q>0$).

\medskip

\subsection{$p$-rotor walk on $\Zb^2$}
Unfortunately, the techniques used in this paper are very sensitive to changes 
to the model.
Indeed, consider the \emph{$p$-rotor walk} on $\Zb^2$~(introduced in \cite{HLSH}),  
which is an RWLM where, for each step, a walker rotates the rotor of its current location 90-degrees counter-clockwise with probability $p$, and rotates the rotor 90-degrees clockwise with probability $1-p$.
Note that we recover the counterclockwise rotor walk when $p=0$, the clockwise rotor walk when $p=1$, and the $\Hor$--$\Ver$ walk with $q=1$ when $p=\frac{1}{2}$.

\smallskip

\begin{conjecture}\label{conjprotor walk recurrence}
	Let $p \in [0,1]$.
	Then, for almost every $\rho$ sampled from the $\IID$ uniform measure  on $\{\Hor,\Ver\}$,
	the corresponding $p$-rotor walk on $\Zb^2$ visits every vertex infinitely many times a.s..
\end{conjecture}

\smallskip

This conjecture was answered positively by Theorem~\ref{thmrecurrence IID} for the case $p=\frac{1}{2}$.
However, the techniques of this paper  break down immediately when $p\neq \frac{1}{2}$,
as the best upper bound we have for  $\Eb |\wt(x,t)|$ (recall \eqref{eqweight}) would then have a linear decay  rather than a quadratic decay (see~\eqref{eqweight max asymptotic}), and we need at least a quadratic decay for the proof of Theorem~\ref{thmrecurrence IID} to work.
Thus developing a better upper bound for  $\Eb |\wt(x,t)|$   would constitute a natural first step in proving Conjecture~\ref{conjprotor walk recurrence}.


\bigskip

\section*{Acknowledgement}

The author would like to thank Lionel Levine and Yuval Peres for their advising
throughout the whole project,  Lila Greco  for performing the simulations for Figure~\ref{figure: simple random walk vs rotor walk}, Tal Orenshtein for pointing us to additional references,
and Peter Li for inspiring discussions. Part of this work was done when
the author was visiting the Theory Group at Microsoft Research, Redmond, and
when the author was a graduate student at Cornell University.
The author would also like to thank the anonymous referee and the editor for valuable comments and
references that greatly improves the  paper.
In particular, the proof of Lemma~\ref{lemreturn probability limit} is greatly simplified thanks to the referee's comment.


\vskip 0.9 cm

\appendix
\section{Proof of Lemma~\ref{lemmonotonicity}}\label{secappendixA}


\begin{proof}[Proof of Lemma~\ref{lemmonotonicity}]
	For every $x,y \in V$, let $N_t(x,y)$ be the total number of departures from $x$ to $y$,
	by the stack walk with turn order $\Oc$,  up to time $t$,
    \[ N_t(x,y) \ := \  \big|\{ s \in \{1,\ldots,t\} \. \mid \.  X_{s-1}^{(i)}= x, \ X_s^{(i)}=y, \  \ \text{ for some } i \in \{1,\ldots,n\} \} \big|\., \]
    and we write \. $N_{\infty}(x,y) \. := \. \lim_{t \to \infty}N_{t}(x,y)$\..
	We denote by $N_t'(x,y)$ and $N_\infty'(x,y)$ the same numbers for the stack walk with turn order 
	$\Oc'$.
	Note that 
	\[ R_t(x) \ = \ \sum_{y \in \Ngh(x)} N_t(x,y); \qquad R_t'(x) \ = \ \sum_{y \in \Ngh(x)} N_t'(x,y).   \]
	Thus it  suffices to show that, for all $t\geq 0$, we have   \. $N_{t} (x,y) \. \leq \. N_{\infty}'(x,y)$ for  all $x,y \in V$.
	
	Suppose to the contrary that the claim is false.
	Let $t$ be the smallest integer such that the claim is false.
	Then there are two consequences.
	Firstly, 
	\begin{equation}\label{eqhotel 1}
	 R_{t-1}(z) \   \leq \ R'_{\infty}(z)  \quad \text{ for all } z \in V\..
	\end{equation}
	Secondly, there exists $x \in V$ such that,  
	\[ \sum_{y \in \Ngh(x)} N_t(x,y)  \quad = \quad   1 \ + \ \sum_{y \in \Ngh(x)} N_{\infty}'(x,y). \] 
	This implies two other consequences. Firstly, 
	\begin{equation}\label{eqhotel 2}
	N_{t-1}(x,y)  \ = \ N_{\infty}'(x,y) \qquad \text{ for all } y \in \Ngh(x)\..
	\end{equation}
	Secondly,  there is at least one walker present at $x$ at the end of the $t-1$-th step of the stack walk with turn order $\Oc$, i.e., 
	\begin{equation}\label{eqhotel 3}
	 |\{i \. \mid \. \Xb^{(i)}=x \} | \ + \  R_{t-1}(x)  \ - \  \sum_{y \in \Ngh(x)} N_{t-1}(x,y)  \quad \geq \quad 1,  
	\end{equation}
	where $\Xb^{(i)} \in V^n$ is the vector that records the initial locations of the walkers.
	
	Plugging \eqref{eqhotel 1} and \eqref{eqhotel 2} into \eqref{eqhotel 3},
	we get 
	\begin{equation*}
	|\{i \. \mid \. \Xb^{(i)}=x \} | \ + \ R_{\infty}(x)   \ - \  \sum_{y \in \Ngh(x)} N'_{\infty}(x,y)  \quad \geq \quad 1,  
	\end{equation*}
	This means that, excepting the first few finite steps,  there is always at least one walker present at $x$  throughout the entirety of the stack walk with turn order $\Oc'$.
	Since $\Oc'$ is regular, this implies that 
	$N_{\infty}'(x,y) = \infty$ for some $y \in \Ngh(x)$, which contradicts  \eqref{eqhotel 2}, as desired.
\end{proof}

\medskip

\section{Proof of Theorem~\ref{theorem: recurrence is invariant under stack relation}}\label{secappendixB}

We now build toward  the proof of Theorem~\ref{theorem: recurrence is invariant under stack relation},
which is adapted  from  \cite[Theorem~1]{AH2}.

Consider a variant of the (multi-walker) stack walk on a finite graph, 
where  the walkers are immediately frozen when they reach a specified set of \emph{sink vertices} $Z$ (which is nonempty).

\smallskip

\begin{lemma}[{Least action principle~\cite[Proposition~4.1]{DF}, see also \cite[Lemma~4.3]{BL}}]\label{lemleast action principle}
	Consider a stack walk on a finite graph with the initial location $\xb \in V^n$,
	the initial regular stack $\xi$, and the nonempty sink $Z$.
	 Then the stack walk terminates in a finite number of moves.
	 The final position, i.e., the number of frozen walkers on each vertex,
	 the total number of moves, and the total number of visits to each vertex, 
	 are independent of the chosen turn order.
\end{lemma}

%

%

\smallskip

We now apply the least action principle to prove increasingly stronger versions of Theorem~\ref{theorem: recurrence is invariant under stack relation}\..

\smallskip

\begin{lemma}\label{lemrecurrence invariant under changed position single walker}
	Let $x,x' \in V$ and let $\xi$ be a regular stack.
	Then the stack walk with a single walker with initialization 
	 $(x,\xi)$ is recurrent if and only if the stack walk with initialization $(x',\xi)$
	is recurrent. 
\end{lemma}

\smallskip

\begin{proof}
	It suffices to prove that if the stack walk started at $x$ is recurrent, 
	then the stack walk  started at a neighbor  $x'$ of $x$ is also recurrent.
		By  \eqref{eqrecurrence} it then suffices to prove that,
	for every $k \geq 0$,
	the stack walk started at $x'$ visits $x$ at least $k$ times.
	

	Let $m:=m(k)$ be the smallest integer such that 
	\. $m\geq k-1$ \.  and $\xi(x,m)=x'$\.. 
	Note that $m$ exists because $\xi$ is regular.
	Let $S$ be the finite set of vertices visited by the stack walk started at $x$
	until it has made  $m+1$ returns to $x$.
	Let $F$ be the finite subgraph of $G$ induced by $S$,
	with  two  additional sink vertices $Z=\{z_1,z_2\}$.
	Let $\xi'$ be the stack of $S$ 
	where  every card in $\xi$ pointing outside of $S$ is  replaced with a card pointing to $z_1$,
	and every card in $\xi$ pointing to $x$ is replaced with a card pointing to $z_2$.
	
	We now start a stack walk on $S$ with $m+1$ walkers at $x$, with initial stack $\xi'$,
	and with sink vertices $Z=\{z_1,z_2\}$.
	We perform the stack walk with the following turn order:
	In the beginning a walker   leaves $x$ and performs stack walk until it reaches $Z$.
	Each time a walker reaches $Z$, we repeat the same process with another walker at $x$  until all walkers reach $Z$.
	Note that this multi-walker stack walk on $S$  follows exactly the trajectory of the original stack walk on $G$ started at $x'$, and  
	all the walkers  in fact reach $z_2$ (since the original walk is recurrent).
	By the least action principle~(Lemma~\ref{lemleast action principle}),
	we can perform this stack walk on $S$ using any turn order, and eventually all walkers will reach $z_2$.

	Now, we choose another turn order for this stack walk on $S$.
	First let all the $m+1$ walkers take one step,
	so there is now one walker at each vertex \. $\xi(x,i)$ for $i \in \{0,\ldots,m\}$ \. (including $x'$).
	Now let the walker at $x'$ perform stack walk until it gets absorbed at $z_2$.
	Whenever the $i$-th walker ($i \in \{1,\ldots,m+1\}$) is absorbed at $z_2$,
	we repeat the same process with the $i+1$-th walker being the walker at \. $\xi(x,i-1)$\.,
	until every walker reaches $z_2$.
	Note that this stack walk on $S$  follow exactly the trajectory of the original stack walk on $G$ started at $x'$, 
	and thus we conclude that the latter visits $x$ at least $m+1 \geq k$ times, as desired.
\end{proof}

\smallskip

\begin{lemma}\label{lemrecurrence invariant under changed position}
	Let $\xb,\xb' \in V^n$ and let $\xi$ be a regular stack.
	Then  
	$(\xb,\xi)$ is recurrent if and only if  $(\xb',\xi)$
	is recurrent. 
\end{lemma}

\smallskip

\begin{proof}
	It suffices to show that if $\xb$ and $\xb'$ differs by exactly one coordinate  and $(\xb,\xi)$ is recurrent, then $(\xb',\xi)$ is recurrent.
	By Lemma~\ref{lemabelian property}, we can without loss of generality assume that $\xb$ and $\xb'$ differ only at the $n$-th coordinate.
	We now consider the stack walk with the $n$-th walker removed.
	That is,  
	let $\xb'' \in V^{n-1}$ be the vector defined by 
	\. $\xb''(i) \. := \. \xb(i) = \xb'(i)$ \. $(i \in \{1,\ldots,n-1\})$\..
	There are two cases to check.
	
	Firstly, suppose that  $(\xb'',\xi)$ is recurrent. In this case, it follows from Lemma~\ref{lemmonotonicity} that $(\xb',\xi)$ is also recurrent.
	
	Secondly, suppose that $(\xb'',\xi)$ is transient. 
	Let \. $(\Xb_t,\xi_t)_{t \geq 0}$ \. be the stack walk with initialization $(\xb'',\xi)$.
	Since the stack walk is transient,
	the stack $\xi'$ given  by  $\xi' \ := \ \lim_{t \to \infty} \xi_t$ is well defined.
	Since $(\xb,\xi)$ is recurrent, it then follows from Lemma~\ref{lemabelian property} that the single-walker stack walk with initialization $(\xb(n), \xi')$ 
	is recurrent.
	By Lemma~\ref{lemrecurrence invariant under changed position single walker} 
	this implies that the stack walk with 
	initialization $(\xb'(n), \xi')$ 
	is recurrent.
	Finally by Lemma~\ref{lemabelian property} again we conclude that $(\xb',\xi)$ is recurrent. 
	This completes the proof.
\end{proof}

\smallskip

\begin{proof}[Proof of Theorem~\ref{theorem: recurrence is invariant under stack relation}]
	It suffices to consider the case when $\xi'$ is obtained from $\xi$ by a single popping operation at $x$, i.e., $\xi'=\varphi_x(\xi)$.
	By Lemma~\ref{lemrecurrence invariant under changed position} we can without loss of generality 
	assume that all the walkers are initially located at one vertex, i.e., \.  $\xb =\xb' =(x,\ldots,x)$\..
	
	Suppose that $(\xb,\xi)$ is recurrent.
	 Let one walker at $x$ performs one step of the stack walk.  
	Then the  pair of walkers-and-stack  changes 
	to \.  $(\xb_1,\xi_1)$, where 
	\[ \xb_1 \ := \ (\xi(x,0),x,\ldots,x); \qquad \xi_1 \ := \ \varphi_x(\xi) \ = \ \xi'\.,  \]	
	and note that $(\xb_1,\xi_1)$ is recurrent by the transitivity of recurrence.
	It now follows from Lemma~\ref{lemrecurrence invariant under changed position} that \. $(\xb',\xi') \. = \. (\xb',\xi_1)$ \. is recurrent, and the proof is complete.
\end{proof}

\vskip 0.9 cm

 \end{document}